\documentclass[12pt,psfig]{article}
\usepackage[toc,page]{appendix}
\usepackage{amsmath, amsthm, amssymb,amsfonts,amscd}
\usepackage{graphicx,epsfig,latexsym}
\usepackage{cite,calc,xcolor,subfigmat,mathrsfs,dsfont}
\usepackage{graphicx,epsfig}
\usepackage{amsfonts}
\usepackage{amscd}
\usepackage{mathrsfs}
\usepackage{enumerate}
\usepackage{latexsym}
\usepackage{lineno}
\usepackage{pst-node}
\usepackage{tikz-cd}

\def\zz{\mathfrak{a}}
\def\del{\mathfrak{b}}
\def\na{a}

\def\nM{\mathsf{M}}
\def\nN{\mathsf{N}}
\def\nH{\mathsf{H}}
\def\nA{\mathsf{A}}

\def\nB{\mathsf{B}}
\def\nT{\mathsf{T}}

\def\nX{\mathsf{X}}

\def\nZ{\mathsf{Z}}

\newcommand{\Sl}{\mathfrak{sl}_2}

\newcommand{{\E}}{{\mathsf E}}
\newcommand{\0}{\mathbf{0}}

\newcommand{\nv}{\mathsf{v}}
\newcommand{\nw}{\mathsf{w}}

\newcommand{\n}{\mathsf{n}}

\newcommand{\D}{{\partial }}
\newtheorem{thm}{Theorem}[section]
\newtheorem{cor}[thm]{Corollary}
\newtheorem{lem}[thm]{Lemma}
\newtheorem{prop}[thm]{Proposition}
\theoremstyle{definition}
\newtheorem{defn}[thm]{Definition}
\newtheorem{exm}[thm]{Example}
\newtheorem{rem}[thm]{Remark}

\numberwithin{equation}{section}
\usepackage{color}
\newcounter{IssueCounter}
\newtheorem{Issue}[IssueCounter]{Issue}
\newcommand{\issue}[2]{
\begin{Issue}[\textcolor{red}{#1}]{\textcolor{blue}{#2}}\end{Issue}}

\def\be {\begin{equation}}
\def\ee {\end{equation}}
\def\ba {\begin{eqnarray}}
\def\ea {\end{eqnarray}}
\def\bpr {\begin{proof}}
\def\epr {\end{proof}}
\def\bes {\begin{equation*}}
\def\ees {\end{equation*}}
\def\bas {\begin{eqnarray*}}
\def\eas {\end{eqnarray*}}

\begin{document}
\renewcommand {\thefootnote}{\dag}
\renewcommand {\thefootnote}{\ddag}
\renewcommand {\thefootnote}{ }

\title{Equivariant decomposition of polynomial vector fields}

\author{
Fahimeh Mokhtari\footnote{Email: fahimeh.mokhtari.fm@gmail.com}
\\Jan A. Sanders \footnote{Email: jan.sanders.a@gmail.com}\\
Department of Mathematics, Faculty of Sciences\\
Vrije Universiteit, De Boelelaan 1081a,\\ 1081 HV Amsterdam, The Netherlands
}

  
\maketitle



\begin{abstract}
To compute the unique formal normal form of families of vector fields with nilpotent linear part,
we choose a basis of the Lie algebra consisting of orbits under the linear nilpotent.
This creates a new problem: to find explicit formulas for the structure constants in this new basis.
These  are well known in the 2D case, and recently expressions were found for the 3D case by ad hoc methods.
The goal of the present paper is to formulate a systematic approach to this calculation.

We propose to do this using
a rational method for the inversion of the Clebsch-Gordan coefficients.
We illustrate the method on a family of 3D vector fields and 
compute the unique formal normal form for the Euler  family both in the 2D and 3D case.

\end{abstract}\vspace{-0.2in}
\vspace{0.10in} 

\noindent {\it Keywords}: Invariant theory; Clebsch-Gordan coefficients; Inversion of the Clebsch-Gordan coefficients; Nilpotent singularity; 
Unique normal form; Euler vector field.
\def\sad{{\rm ad}}
\def\mf{\mathfrak}
\section{Introduction}\label{sec:1}
	 Consider formal vector fields in  \(\mathbb{R}^2\) with a fixed nilpotent linear part \(\nN\). In \cite{baider1992further} the following problem was addressed:
for each such vector field, determine the unique formal normal form, that is a form in which the vector field can be transformed by formal transformations such that the only transformations that are still possible without perturbing the normal form style are formal symmetries of the vector field.
		
This was done by constructing a new basis for the formal vector fields in such a way if \(\nv^i\) is a basis vector, so is \(\nv^{i+1}=[\nN,\nv^i]=\sad_\nN \nv^i\).
The inversion of \(\sad_\nN\) is then a simple matter of changing the appropriate index from \(i\) to \(i-1\), if \(i>0\).
Constructing such a basis is not too difficult in lower dimensions. The problem is to compute the structure constants of the Lie algebra in this new basis.
This is the equivariant decomposition of the vector fields we refer to in the title. 

In this paper, we aim to generalize this procedure from \(\mathbb{R}^2\) to \(\mathbb{R}^n\).
Before the reader gets too excited about this, we should point out that we do not aim to solve the unique formal normal form problem for \(n\)-dimensional nilpotent systems.
The fact that even the \(2\)D results are not complete may indicate the size of such an undertaking.
What we want to obtain here is a systematic procedure to compute the structure constants in a basis given by orbits of the nilpotent.
This is a first step towards the determination of the unique formal normal form.

Besides from having a theoretical interest (the classification of formal vector fields
under formal coordinate transformations) we should point out a practical consequence 
of this analysis. If one does not know what the unique formal normal form is then it is difficult
to determine in a bifurcation analysis which parameters are essential and which are not (because some could be eliminated in the normalization procedure). This leads to confusion in the process of automated (numerical) bifurcation analysis.
\begin{rem}
In the sequel, we will occasionally drop the {\em formal}, but it should be clear that there are no results in this paper beyond the formal normal form.
\end{rem}
One might ask, why go to the \(3\)D case if the \(2\)D case is not done yet? To give at least a partial response to this question
we analyze a \(2\)D subcase, the Euler vector fields, that was left open in op. cit., in Section \ref{2DBT}. 
We will see that the \(3\)D result is very similar to the \(2\)D result. This is a phenomenon that is familiar from Classical Invariant Theory.
One can expect that the \(4\)D analysis will be much harder (but then the \(5\)D case should be similar to the \(4\)D).

The complexities of this programme are already clear in the treatment of the \(3\)D irreducible nilpotent singularity in \cite{JF2019vector}.
There the structure constants are computed using Maple computations and extrapolation, a time consuming method which does not seem to be well suited for higher dimensions.
In this paper, we will formulate a procedure to compute the structure constants using two ideas:
\begin{itemize}
\item
Use the Clebsch-Gordan coefficients to express two-tensors in terms of \(\nN\) orbits of transvectants and apply this both in the symmetrical case to products of functions and in the
antisymmetrical case to Lie brackets.
\item
Compute the transvectants, using the fact that we know how to express the elements in \(\ker\sad_\nM\) using a Stanley basis \cite[Chapter 12]{SVM2007}.
This last condition restricts at the moment the dimension \(n\) to be less than \(6\), at least in the irreducible case.
\end{itemize}
We hope that these techniques will enable the computation of the structure constants for the basis of \(\nN\)-orbits for higher dimensions,
although we realize that this will probably need some more ideas in order to be completely carried out (the difficulty being the lack of explicit formulas for the transvectants).
We mention here that if one is only interested in carrying out this program for a  finite jet degree, there should be no difficulty in higher dimensions; it is the infinite-dimensional character
of the problem and our need for explicit formulas that is causing  difficulties.

For the moment the \(3\)D programme of computing the unique normal form with nilpotent singularity will keep us busy for a while, if the experience with the \(2\)D
case, where there are still subcases that are not completely understood, is any indication.

The first step is to construct an \(\mathfrak{sl}_2\)-triple 
\(\langle \nN, {{\mathsf H}}, \nM\rangle\),
 where \({{\mathsf H}}\) and \(\nM\) are also linear vector fields, using the Jacobson-Morozov Lemma (see \cite{knapp2013lie,collingwood1993nilpotent}). 
The Lie bracket relations are:
\ba\label{eq:sl2}
\nonumber
[\nM,\nN]&=&\nH,\\
{[}\nH,\nN]&=&-2\nN,\\
\nonumber
{[}\nH,\nM]&=&2\nM.
\ea
This triple is not uniquely determined, and one can  make choices here that are most suitable for the given situation.
This construction is possible when the linear vector fields form a reductive Lie algebra, as is the case for \(\mf{gl}_n\) and \(\mf{sp}_n\).

\def\im{\mathrm{im}\ }
A formal vector field is said to be in {\em first level normal form in the \(\Sl\)-style} if its nonlinear part is contained in \(\ker \sad_\nM\).
The reason for this is that the finite-dimensional representation theory of \(\Sl\) teaches us that a finite-dimensional \(\Sl\)-module \(V\) is isomorphic (as \(\Sl\)-modules, not as Lie algebras) to  \(\im \sad_\nN\oplus\ker \sad_\nM\) and this goes through in the infinite-dimensional formal vector field setting.

The next step is to define a basis of vector fields of a given degree \(d\) by first giving a basis for the elements in \(\ker\sad_\nM\) of that degree, say \(\{\nv_{d,1},\cdots,\nv_{d,k}\}\) 
	and then adding to the general basis the elements \(\nv_{d,k}^i=\sad_\nN^i \nv_{d,k}, k=1,\cdots,i, i=1,\cdots,\infty\).
Despite the \(\infty\) in this definition, this is a finite set due to the nilpotent character of \(\nN\).
This way, it is easy to see what the preimage under \(\sad_\nN\) is, but the problem is now that one does not have immediately explicit formulas for the structure constants of the Lie algebra with this basis.

\subsection{Outline}
What we want to accomplish in this paper is the following: suppose we have two polynomial vector fields, say \(\sad_\nN^k v \) and \(\sad_\nN^l w\), with \(v,w\in\ker\sad_\nM\).
Can we express the Lie bracket \([\sad_\nM^k v,\sad_\nM^l w]\) in a similar form, that is, as a linear combination of elements of the form \(\sad_\nN^{m_i} u_i, u_i\in\ker\sad_\nM\)?
Our approach will be to find a formula that expresses \(\sad_\nN^k v \otimes \sad_\nN^l w\) as  a linear combination of \(\sad_\nN^{m_i} u_i\), where the \(u_i\in\ker\sad_\nM\) are {\em transvectants}.
Transvectants were invented to be used in Classical Invariant Theory to produce from one or more given covariants a new covariant.
If this sounds like something one could do forever, the results of Gordan and, somewhat more convincingly, Hilbert showed that this is in fact a finite procedure and at some point all the generating
covariants are found. In our context, covariants are associated to the elements in \(\ker\sad_\nM\).
These transvectants are linear combinations of terms of the type \(\sad_\nN^k v \otimes \sad_\nN^l w\) and it might seem that we are running around in circles.
However, the fact that the transvectants are in the kernel of \(\sad_\nM\) gives us a way to explicitly compute them, at least in the \(3\)D case that we will be most interested in.

We then apply these formulas to compute the unique formal normal form for general Eulerian vector fields with given nilpotent linear term in several cases.
\subsection*{Acknowledgments}The authors would like to express their thanks to  professors
Christian Kassel and Tom Koornwinder, for the email exchanges in the process of getting the inversion of the Clebsch-Gordan coefficients right,
and Bob Rink, for reading the manuscript and providing us with valuable suggestions.

\section{Clebsch-Gordan symbols, a rational approach}\label{sec:CGS}

We now turn from the polynomial vector fields to a more abstract approach involving \(\Sl\)-modules. Here we write \(\nN\) for \(\sad_\nN\) to denote the \(\Sl\)-action.

The following is based on \cite[Chapter XIV]{kassel2012quantum}, to which we refer to for some of the technical details taken for granted in the the sequel.
There an explicit method is given to express elements in \(V_m\otimes V_n\) in terms of the \(\nN\)-orbits of their transvectants (these are given in Lemma VII.7.2, loc. cit.).
The only problem with this formula is that it implicitly assumes the norm of the transvectant to be \(1\).
Here we present a corrected version of this, restricted to the classical case \(q=1\), the quantum case should be completely similar.
The rational approach that is followed, see also \cite{rationalCGC2018}, does avoid the square roots that appear in the usual quantum mechanical approach,
where one works with an orthonormal base \cite{van2012group,koornwinder1981clebsch,kirillov1988quantum}; we are happy with orthogonality.
It may also be useful in the categorification programme, see \cite{frenkel2010categorifying}.
\begin{rem}
There is some historical confusion in the use of the word orthogonal in linear algebra, where the term orthogonal matrix is used for a matrix with orthonormal column vectors. When we use orthogonal we mean orthogonal and when we use orthonormal we mean orthonormal.
\end{rem}
Let \(U,V,W\) be finite-dimensional \(\Sl\)-modules. The reader may want to think of polynomial vector fields of some bounded degree or polynomials,
where \(\nN\) is given, and the Jacobson-Morozov Lemma is invoked to obtain the \(\Sl\)-action.
In the sequel we use \(V\) to describe some properties of \(\Sl\)-modules, but this could as well be \(U\) or \(W\).

We know \(V\) to be a direct sum of finite-dimensional \(\mathsf{H}\)-eigenspaces \(V_m\), with \(\dim V_m=m+1\)
where each \(V_m\) is generated by the \(\nN\)-action on a \(\nv_m\in\ker\nM\) with \(\nH \nv_m=m\nv_m\).
We choose a basis for \(\ker \nM\) consisting of eigenvectors, say
 \[\langle \nv^{\{1\}}_{m_1},\ldots,\nv^{\{k_{m_1}\}}_{m_1},\ldots,\nv^{\{1\}}_{m_k},\ldots,\nv^{\{k_{m_k}\}}_{m_k}\rangle.\]
If there is not a given inner product on \(\ker \nM\) we construct the standard dual basis \[\langle \nv^{\{1\star\}}_{m_1},\ldots,\nv^{\{k_{m_1}\star\}}_{m_1},\ldots,\nv^{\{1\star\}}_{m_k},\ldots,\nv^{\{k_{m_k}\star\}}_{m_k}\rangle,\]
and define the standard inner product \(( \nv^{\{i\}}_k, \nv^{\{j\}}_l)_{\ker \nM|V}=\nv^{\{i\star\}}_k(\nv^{\{j\}}_l)=\delta_{i,j}\delta_{k,l}\), with linear extension defined on all of \(\ker \nM\). 

Then we let for given \(\nv_m \in\ker \nM \) with \(\mathsf{H}\)-eigenvalue \(m\),
\bas
\nv^i_m= \nN^i \nv_m,  \nv_m\in\ker \nM, \quad \nv_m^{(i)}=\frac{1}{i!} \nv_m^i, \nN \nv_m^{(i)}=\frac{1}{i!} \nN \nv_m^i=(i+1) \nv^{(i+1)}_{m}. 
\eas
We denote the space spanned by \(\nv_m^i, i=0,\ldots,m\) by \(V_m\).
It follows from standard \(\Sl\)-representation theory that this way we can construct a basis for \(V\).

We then define an induced inner product on \( V\) by defining it on eigenvectors by 
\[ (\nv_m^{(i)},\nw_n^{(j)})_V=\binom{m}{i}\delta_{i,j}\delta_{m,n}(\nv_m,\nw_n)_{\ker \nM|V},\quad \nv_m,\nw_n\in\ker \nM|V,\]
and then by linear extension on all of \(V\). The \(\delta_{m,n}\) is not really necessary here, as it follows from the definition of \((\cdot,\cdot)_{\ker \nM|V}\),
but it is used to emphasize the orthogonality of the \(v_m^{(i)}\).

\begin{lem} 
One has \(\nN^\top=\nM\) and \(\nM^\top=\nN\), where \(\nN^\top\) is the transpose of \(\nN\) with respect to the inner product \((\cdot,\cdot)_V\):
\[
( \nN \nv , \nw )_V =(\nv, \nM \nw)_V, \quad \nv,\nw\in V.
\]
\end{lem}
\begin{proof}
We compute the two sides on basis vectors:
\bas
(\nN \nv_m^{(i)},\nw_n^{(j)})_V&=& ( (i+1) \nv_{m}^{(i+1)},\nw_n^{(j)})_V=(i+1)\binom{m}{i+1}\delta_{i+1,j}\delta_{m,n}(\nv_m,\nw_n)_{\ker \nM}
\\&=&(m-i)\binom{m}{i}\delta_{i+1,j}\delta_{m,n}(\nv_m,\nw_n)_{\ker \nM},\\
(\nv_m^{(i)}, \nM \nw_n^{(j)})_V&=& (\nv_m^{i}, \frac{1}{j!}\nM \nN^j \nw_{n})=(\nv_m^{(i)},\frac{1}{(j-1)!} (m-j+1) \nN^{j-1} \nw_{n})_V\\&=&(m-j+1) (\nv_m^{(i)}, \nw_{n}^{(j-1)})_V
= (m-i)  \binom{m}{i} \delta_{i,j-1}\delta_{m,n} (\nv_m,\nw_n)_{\ker \nM}.
\eas
By linear extension this holds  on \(V\).
\end{proof}
\begin{rem}
This could also be formulated as: let \((\cdot,\cdot)_V\) be an inner product on \(V\) such that \(\nN^\top=\nM\).
The defining property can then be readily derived.
\end{rem}

For later use we state:
\begin{prop}\label{prop:binom}
The following binomial identity holds:
\bas
\sum_{i+j=p} \binom{m-i}{m-p} \binom{n-j}{n-p}=\binom{m+n-p+1}{p}.
\eas
\end{prop}
\begin{proof}
Cf. \cite{rationalCGC2018}.
\end{proof}
\begin{defn}
We define the \(p\)th transvectant of \(\nv_m \otimes \nw_{n}\), \[\Join_{m+n-2p}: V_m\otimes W_n \rightarrow U_{m+n-2p}\cap \ker \nM,\] as
\bas
\Join_{m+n-2p}\nv_m \otimes \nw_{n}=
\sum_{i+j=p}  (-1)^i \binom{p}{i}   \frac{\nv^{(i)}_{m}}{\binom{m}{i}}\otimes \frac{\nw^{(j)}_{n}}{\binom{n}{j}}.
\eas
This definition immediately follows from the symbolic method from Classical Invariant Theory, see \cite{olver1999classical}.
Alternative definitions are also in use, with multiplication constants depending on \(m,n,p\), cf Remark \ref{rem:trans}.
The notation with the \(\Join\) is nonstandard, but it reflects the transvection process so nicely that we do not doubt Sylvester would have approved of it.
\end{defn}

\begin{rem}\label{rem:trans}
We compare the definition given in \cite{SVM2007} with the present definition of transvectant:
\bas
(\nv,\nw)^{(p)}&=& p! \sum_{i+j=p} (-1)^i \binom{m-i}{p-i}\binom{n-j}{p-j} \nv_m^i \otimes \nw_n^j
\\&=& \frac{1}{(m-p)!(n-p)!}\sum_{i+j=p} (-1)^i \frac{\binom{p}{i}}{\binom{m}{i}\binom{n}{j}}   \nv_m^{(i)} \otimes \nw_n^{(j)}
\\&=&\frac{1}{(m-p)!(n-p)!}
\Join_{m+n-2p} \nv_{m}\otimes \nw_{n}.
\eas
\end{rem}
\begin{lem}
In order to implement the procedure given in {\em \cite{kassel2012quantum}} correctly, we compute
\bas
\|\Join_{m+n-2p} \nv_m \otimes \nw_n\|^2&=&
 \frac{\binom{m+n-p+1}{p}}{\binom{m}{p}\binom{n}{p}}
 \|\nv_m\otimes \nw_n\|^2.
\eas
\end{lem}
\begin{proof}

One has
\bas
\lefteqn{(\Join_{m_1+n_1-2p_1} \nv_{m_1}\otimes \nw_{n_1},\Join_{m_2+n_2-2p_2} \nv_{m_2}\otimes \nw_{n_2})_{V\otimes W}}&&\\&=&\sum_{i_1=0}^{p_1}\sum_{i_2=0}^{p_2} (-1)^{i_1+i_2} 
\frac{\binom{{p_1}}{i_1}}{\binom{m_1}{i_1}\binom{n_1}{p_1-i_1}}\frac{\binom{{p_2}}{i_2}}{\binom{m_2}{i_2}\binom{n_2}{p_2-i_2}} 
(\nv_{m_1}^{(i_1)},\nv_{m_2}^{(i_2)})_{V}(\nw^{(p_1-i_1)}_{n_1},\nw^{(p_2-i_2)}_{n_2})_{W}
\\&=&\delta_{p_1,p_2}\delta_{m_1,m_2}\delta_{n_1,n_2}\|\nv_m\|^2 \|\nw_n\|^2
\sum_{i=0}^{p} \frac{\binom{{p}}{i}^2}{\binom{m}{i}\binom{n}{p-i}}  ,
\eas
and it follows that the transvectants are orthogonal and that
\bas
\|\Join_{m+n-2p} \nv_{m}\otimes \nw_{n}\|^2&=&\|\nv_m\|^2 \|\nw_n\|^2\sum_{i=0}^p \frac{\binom{p}{i}^2}{\binom{m}{i}\binom{n}{p-i}}
\\&=&\|\nv_m\|^2 \|\nw_n\|^2\frac{p!^2}{m!n!}\sum_{i=0}^p \frac{ (m-i)!  (n-p+i)! }{i!(p-i)!}.
\eas
Using Proposition \ref{prop:binom}
we find
\bas
\sum_{i+j=p} \frac{ (m-i)!  (n-j)! }{i!j!} &=&\sum_{i+j=p} \binom{m-i}{m-p}\binom{n-j}{n-p} \frac{ (m-p)! (p-i)! (n-p)! (p-j)!}{i!j!}
\\&=& (m-p)!  (n-p)! \sum_{i+j=p} \binom{m-i}{m-p}\binom{n-j}{n-p}
\\&=& (m-p)!  (n-p)! \binom{m+n-p+1}{p},
\eas
and it follows that
\bas
\|\Join_{m+n-2p} \nv_m \otimes \nw_n\|^2&=&\|\nv_m\|^2 \|\nw_n\|^2\frac{p!^2}{m!n!}\sum_{i=0}^p \frac{ (m-i)!  (n-p+i)! }{i!(p-i)!}
\\&=&\frac{\binom{m+n-p+1}{p}}{\binom{m}{p}\binom{n}{p}}\|\nv_m\|^2 \|\nw_n\|^2.
\eas
\end{proof}
\begin{cor}
\bas
\|\Join_{m+n-2p}^{(k)} \nv_m \otimes \nw_n\|^2&=&
 \frac{\binom{m+n-2p}{k}\binom{m+n-p+1}{p}}{\binom{m}{p}\binom{n}{p}}\|\nv_m\otimes \nw_n\|^2.
\eas
\end{cor}

Although we have already computed \(\|\Join_{m+n-2p}^{(k)} \nv_m \otimes \nw_n\|^2\),
we have not computed \(\Join_{m+n-2p}^{(k)} \nv_m \otimes \nw_n\) itself.
We give here an explicit derivation, but mention that it also immediately follows from the symbolic method.
Let \(\Delta(X)=X\otimes 1 +1\otimes X\).
\bas
\Join_{m+n-2p}^{(k)} \nv_m \otimes \nw_n&=&\frac{1}{k!}\Delta^k(\nN) \Join_{m+n-2p} \nv_{m}\otimes \nw_{n}
\\&=&
\frac{1}{k!}\Delta^k(\nN) 
\sum_{i=0}^p (-1)^i \frac{\binom{p}{i}}{\binom{m}{i}\binom{n}{p-i}}  \nv_{m}^{(i)} \otimes \nw^{(p-i)}_{n}
\\&=&
\frac{1}{k!}
\sum_{r=0}^k \sum_{i=0}^p (-1)^i \frac{\binom{p}{i}}{\binom{m}{i}\binom{n}{p-i}} \binom{k}{r} \frac{(i+k-r)!}{i!}\frac{(p-i+r)!}{(p-i)!} v^{(i+k-r)}_{m} \otimes \nw^{(p-i+r)}_{n}
\\&=&
\sum_{r=0}^k \sum_{i=0}^p (-1)^i \frac{\binom{p}{i}}{\binom{m}{i}\binom{n}{p-i}} \binom{i+k-r}{k-r} \binom{p-i+r}{r} v^{(i+k-r)}_{m} \otimes \nw^{(p-i+r)}_{n}
\\&=&
\sum_{r=0}^k \sum_{i=k-r}^{k-r+p} (-1)^{i-k+r} \frac{\binom{p}{i-k+r}}{\binom{m}{i-k+r}\binom{n}{p+k-i-r}} \binom{i}{k-r} \binom{p+k-i}{r} \nv^{(i)}_{m} \otimes \nw^{(p+k-i)}_{n}
\\&=&
\sum_{i+j=k+p}\sum_{r+q=k} (-1)^{i-k+r} \frac{\binom{p}{i-q}\binom{i}{q} \binom{j}{r}}{\binom{m}{i-q}\binom{n}{j-r}}  \nv_{m}^{(i)} \otimes \nw_{n}^{(j)}.
\eas
In all formulas, the summation convention is such that for a binomial expression we must have that the numbers \(a\) and \( b\) in \(\binom{a}{b}\) are natural numbers (zero included)
with \(a\geq b\).
\begin{defn}
Define the Clebsch-Gordan Coefficient, also known as \(3j\)-symbol , by
\[\begin{pmatrix} m& n& m+n-2p\\i&j& k\end{pmatrix}=\sum_{r+q=k} (-1)^{i-k+r} \frac{\binom{p}{i-q}\binom{i}{q} \binom{j}{r}}{\binom{m}{i-q}\binom{n}{j-r}}, \quad i+j=k+p,\]
	with \(j_1=m, j_2=n\) and \(j_3=n+m-2p\) as the three \(j\)'s. 
\end{defn}
\begin{lem}
The Clebsch-Gordan Coefficients are orthogonal:
\bas
\lefteqn{\sum\begin{pmatrix} {m_1}& {n_1}& {m_1}+{n_1}-2p_1\\i_1&j_1& k_1\end{pmatrix}
\begin{pmatrix} {m_2}& {n_2}& {m_2}+{n_2}-2p_2\\i_2&j_2& k_2\end{pmatrix} \binom{{m_1}}{i_1} \binom {{n_1}}{j_1}}&&
\\&=&
\delta_{k_1,k_2}\delta_{m_1,m_2}\delta_{n_1,n_2} \binom{m_1+n_1-2p_1}{k_1}
 \frac{\binom{m_1+n_1-p_1+1}{p_1}}{\binom{m_1}{p_1}\binom{n_1}{p_1}}
\|\nv_{m_1}\otimes \nw_{n_1}\|^2
\eas
where the summation is over \(i_1+j_1=p_1+k_1=p_2+k_2=i_2+j_2.\)
\end{lem}
\begin{proof}
We compute 
\bas
\lefteqn{(\Join_{m_1+n_1-2p_1}^{(k_1)} \nv_{m_1}\otimes  \nw_{n_1} , \Join_{m_2+n_2-2p_2}^{(k_2)} \nv_{m_2}\otimes  \nw_{n_2} )_{V\otimes W}}&&
\\&=&
\delta_{k_1,k_2} \binom{m_1+n_1-2p_1}{k_1}(\Join_{m_1+n_1-2p_1} \nv_{m_1}\otimes  \nw_{n_1} , \Join_{m_2+n_2-2p_2} \nv_{m_2}\otimes  \nw_{n_2} )_{\ker\Delta(\nM)|V\otimes W}
\\&=&
\delta_{k_1,k_2}\delta_{m_1,m_2}\delta_{n_1,n_2}  \binom{m_1+n_1-2p_1}{k_1}
 \frac{\binom{m_1+n_1-p_1+1}{p_1}}{\binom{m_1}{p_1}\binom{n_1}{p_1}}
\|\nv_{m_1}\otimes \nw_{n_1}\|^2.
\eas
\end{proof}
\begin{thm}[Rational Clebsch-Gordan Coefficients]\label{thm:RCGC}

The inversion formula to the definition of \(\Join_{m+n-2p}^{(k)}\nv_m\otimes \nw_n\) is
\bas
\nv^{(i)}_m\otimes \nw^{(j)}_n&=&\sum_{p+k=i+j}\begin{pmatrix} m&n&m+n-2p\\i&j&k\end{pmatrix}\frac{ \binom{m}{i}\binom{n}{j}\binom{m}{p}\binom{n}{p}}{\binom{m+n-2p}{k}\binom{m+n-p+1}{p}}
\Join_{m+n-2p}^{(k)}\nv_m\otimes \nw_n.
\eas
\end{thm}
\begin{rem}\label{rem:pis0}
We compute the special case \(p=0\).
\[\begin{pmatrix} m& n& m+n\\i&j& i+j\end{pmatrix}=\sum_{r+q=i+j} (-1)^{-j+r} \frac{\binom{0}{i-q}\binom{i}{q} \binom{j}{r}}{\binom{m}{i-q}\binom{n}{j-r}}=1,\]
since \(i-q+j-r=0\), and therefore \(q=i\) and \(r=j\).
\end{rem}

\begin{rem}\label{rem:bright}
Another special case is \(i=0\). We find
\[\begin{pmatrix} m& n& m+n-2p\\0&j& k\end{pmatrix}= \frac{ \binom{j}{k}}{\binom{n}{p}},\]
and it follows that
\bas
\nv^{(0)}_m\otimes \nw^{(j)}_n&=&\sum_{p+k=j} \frac{ \binom{j}{k}\binom{n}{j}\binom{m}{p}}{\binom{m+n-2p}{k}\binom{m+n-p+1}{p}}
\Join_{m+n-2p}^{(k)}\nv_m\otimes \nw_n
\\&=&\sum_{p=0}^j \frac{ \binom{j}{p}\binom{n}{j}\binom{m}{p}}{\binom{m+n-2p}{j-p}\binom{m+n-p+1}{p}}
\Join_{m+n-2p}^{(j-p)}\nv_m\otimes \nw_n.
\eas
\end{rem}
\begin{rem}
This theorem, or variants of it, is well known in quantum mechanics, and in the context of quantum groups.
In the context in which we plan to use it, we have no previous examples of its use.
\end{rem}
\begin{proof}
We consider the equation
\[
A_{[r]}\mathbf{x}=\mathbf{y},
\]
where 
\[
\mathbf{x}=\begin{pmatrix}
\nv_{m}\otimes \nw_{n}^{(r)}\\
\nv^{(1)}_m\otimes \nw^{(r-1)}_{n}\\
\vdots\\
\nv_{m}^{(r-1)}\otimes \nw_{n}^{(1)}\\
\nv_{m}^{(r)}\otimes \nw_n\\
\end{pmatrix},\,\,
\mathbf{y}=\begin{pmatrix} 
\Join_{m+n}^{(r)} \nv_m\otimes \nw_n \\
\Join_{m+n-2}^{(r-1)} \nv_m\otimes \nw_n \\
\vdots
\\
\Join_{m+n-2r+2}^{(1)} \nv_m\otimes \nw_n \\
\Join_{m+n-2r} \nv_m\otimes \nw_n 
\end{pmatrix}.
\]
We first make this into an orthonormal set (orthogonality is clear from the definition of the inner product) by dividing each  term in \(x\) and \(y\) by its norm.
Thus the \(i+1\)st column of \(A_{[r]}\) get multiplied by the \(\|\nv^{(i)}_m\otimes \nw^{(r-i)}_{n}\|\)
and the \(p+1\)st row divided by \(\|\Join_{m+n-2p}^{(r-p)} \nv_m\otimes \nw_n\|\).
The resulting matrix is then transposed and we go back to the original bases by doing the same thing again (this is because
the roles of \(x\) and \(y\) are now switched, and therefore multiplication is switched with division).
We recall that \[\|\nv^{(i)}_m\otimes \nw^{(j)}_{n}\|^2=\binom{m}{i}\binom{n}{j}\|\nv_m\otimes \nw_n\|^2,\] and \[\|\Join_{m+n-2p}^{(k)} \nv_m\otimes \nw_n\|^2= \frac{\binom{m+n-2p}{k}\binom{m+n-p+1}{p}}{\binom{m}{p}\binom{n}{p}}
\|\nv_m\otimes \nw_n\|^2,
\] with \(i+j=p+k=r\).
\end{proof}
\begin{rem}
Nice as this inversion formula for the Clebsch-Gordan coefficients may look, the reader should realize that the right-hand side contains expressions that are computable,
but may not always have a nice general formula. On the bright side, we notice that in the normal form computations we often need special cases with very simple \(3j\)-symbols, cf. Remark \ref{rem:bright}.
In physics problems, it would seem the finite-dimensional results are enough,
but in the computation of the structure constants of the Lie algebra of polynomial vector fields with a basis consisting of \(\nN\)-orbits, we need general formulas.
\end{rem}
\begin{rem}\label{rem:contraction}
	Although the transvectants are formulated as a map from the tensor product \(V_m\otimes V_n\) to itself, in the actual application one often considers contractions. For instance, if \(V_m \) and \(V_n\) have polynomial elements, one can take the product of two polynomials,
	defining a contraction \(\pi_\circ : V\otimes V \rightarrow V\).
	If the transvectant is in \(\ker\Delta(\nM)\), then the transvectant after this multiplication has
	been carried out is in \(
\ker\nM\), since \(\nM fg=\nM \pi_\circ f\otimes g=\pi_\circ\Delta(\nM)f\otimes g\).
	Similar reasoning can be applied when \(\pi_\wedge X\otimes Y=[X,Y]\), due to the Jacobi identity.
\end{rem}

\section{Product formula for the 3D-family}
Let
\bas
\nN &:=&{x}\frac{\D}{\D {y}}+2 {y} \frac{\D}{\D {z}}, \quad{\nH}=2 {z} \frac{\D}{\D {z}} - 2{x} \frac{\D}{\D {x}},\qquad \nM:= {z} \frac{\D}{\D {y}}+ 2 {y} \frac{\D}{\D {x}},
\quad \E= 
{x} \frac{\D}{\D {x}}
+{y} \frac{\D}{\D {y}} 
+{z} \frac{\D}{\D {z}},
\eas
where these operators obey the relations \eqref{eq:sl2}.

\begin{defn}
Let \(\zz\) be the linear form in \(x,y,z\) that is in \(\ker\nM\) and has \(\nH\)-eigenvalue \(2\),
that is \(\zz=z\) with our choice of operators. 
Let \(\del\) be the quadratic form in \(\ker\nM \cap \ker\nH\), that is \(\del=\pi_\circ\Join_{0} \zz\otimes \zz\), which comes down to  \(\del=xz-y^2\) with our choice of operators. 
\end{defn}
\begin{thm}\label{thm:2DkerM}
The kernel of \(\nM\), acting on \(\mathbb{R}[x,y,z]\), is \(\mathbb{R}[\zz,\del]\).
\end{thm}
\begin{proof}
See \cite[Section 12.6.2]{SVM2007}.
\end{proof}
In this section, we compute the \(p\)th transvectant of \(\zz^{\mu_1}\) and \(\zz^{\mu_2}\), that is \(\Join_{2{\mu_1}+2{\mu_2}-2p} \zz^{\mu_1}\otimes \zz^{\mu_2}\),
where \(\mu_1,\mu_2\in\mathbb{N}\) .
\begin{defn}\label{def:lambda}
We define, with \(m_i=2\mu_i, i=1,2\), \(p=2\rho\), and \(l_1, l_2\in\mathbb{N}_0\),
\ba\label{eq:lambda}
\nonumber
\lambda^{l_1,\mu_1}_{l_2,\mu_2}(\rho)&=&\begin{pmatrix} {m_1}&{m_2}&{m_1}+{m_2}-2p\\{l_1}&{l_2}&l_1+l_2-p\end{pmatrix}\frac{ \binom{{m_1}}{{l_1}}\binom{{m_2}}{{l_2}}}{\binom{{m_1}+{m_2}-2p}{l_1+l_2-p}\binom{{m_1}+{m_2}-p+1}{p}\binom{p}{\rho}}
\\&\times&{2}^{p}{\mu_1+\mu_2-\rho\choose \rho}{\mu_1\choose \rho}{\mu_2\choose \rho}\frac{(l_1)!(l_2)!}{(l_1+l_2-p)!}.
\ea
\end{defn}
\begin{rem}\label{eq:=lambdal0}
	Following Remark \ref{rem:bright} we have the following simplified formula of  \eqref{eq:lambda} for \(l_1=0,\)
\bas
\lambda^{0,\mu_1}_{l_2,\mu_2}(\rho)&=&
 {2}^{p} p!\frac{  {\mu_1\choose \rho}{\mu_2\choose \rho}{\binom{l_2}{p}}\binom{m_2-p}{m_2-l_2}{\mu_1+\mu_2-\rho\choose \rho}}{\binom{p}{\rho}\binom{{m_1}+{m_2}-2p}{l_2-p}\binom{{m_1}+{m_2}-p+1}{p}}.
\eas
\end{rem}
\begin{thm}
	The following formula holds for the product:
\ba\label{eq:nn}
\nN^{l_1} \zz^{\mu_1} \cdot \nN^{l_2} \zz^{\mu_2}&=& \sum_{p+k={l_1}+{l_2}} \lambda^{l_1,\mu_1}_{l_2,\mu_2}(\rho) \nN^k \zz^{{\mu_1}+{\mu_2}-p}\del^{\rho}.
\ea
\end{thm}
\begin{proof}
We know from Remark \ref{rem:contraction} that \(\pi_\circ\Join_{2{\mu_1}+2{\mu_2}-2p} \zz^{\mu_1}\otimes \zz^{\mu_2}\) is in the kernel of \(\nM\) after contraction of the tensor product (since a tensor stays in the kernel if its arguments are switched and thus the symmetrized tensor is also in the kernel) and is therefore, a sum of monomials \(\zz^Q \del^R\) as follows from Theorem
\ref{thm:2DkerM}.
Let \(m_i=2\mu_i, i=1,2\) be the \({{\mathsf H}}\)-eigenvalues of \(\zz^{\mu_i}\).

Comparing the degrees we see that \({\mu_1}+{\mu_2}=Q+2R\), while comparing the \({{\mathsf H}}\)-eigenvalues we see that \({m_1}+{m_2}-2p=2Q\).
Thus we have \(Q={\mu_1}+{\mu_2}-p\) and \(2R=p\). This implies that \(p\) is even and we put \(p=2\rho\). Then \(\rho=R\) and we know that the transvectant is 
\( c \zz^{{\mu_1}+{\mu_2}-2\rho} \del^\rho\) if \(p\) is even and zero otherwise.
If we evaluate \( c \zz^{{\mu_1}+{\mu_2}-2\rho} \del^\rho\) for \(x=0\), we find \(c \zz^{{\mu_1}+{\mu_2}-2\rho} \del^p\equiv c (-1)^\rho \zz^{{\mu_1}+{\mu_2}-2\rho} y^{2\rho}\).
Evaluating at \(x=0\) means that we can compute with \(\nN_{app}=2y\frac{\partial}{\partial z}\) since the contribution of the \(x\frac{\partial}{\partial y}\)-term
ends up as zero anyway.

 Let \(\nv_{{m_1}}=\zz^{\mu_1}\) and \(\nw_{{m_2}}=\zz^{\mu_2}\). 
Recall the definition of the transvectant:
\bas
\Join_{m+n-2p}\nv_m \otimes \nw_{n}&=&
\sum_{i=0}^p (-1)^i \frac{\binom{p}{i}}{\binom{m}{i}\binom{n}{p-i}}  \nv_{m}^{(i)} \otimes \nw_{n}^{(p-i)}.
\eas
Then,
\bas\label{T2}
                                \pi_\circ\Join_{{m_1}+{m_2}-2p} \zz^{\mu_1}\otimes \zz^{\mu_2}&=&
                                \pi_\circ\sum_{i=0}^{p} (-1)^i \frac{\binom{p}{i}}{\binom{{m_1}}{i}\binom{{m_2}}{p-i}}  \nv_{m_1}^{({i})} \otimes \nw_{m_2}^{({p-i})}
\\&=&
                                \pi_\circ\sum_{i=0}^{p} \frac{\binom{p}{i}}{\binom{{m_1}}{i}\binom{{m_2}}{p-i}}\frac{(-1)^i}{i!(p-i)!}  \nN^i \zz^{{\mu_1}} \otimes \nN^{p-i}\zz^{{\mu_2}}
\\&\equiv&
                                \pi_\circ\sum_{i=0}^{p} \frac{\binom{p}{i}}{\binom{{m_1}}{i}\binom{{m_2}}{p-i}}\frac{(-1)^i}{i!(p-i)!}  \nN_{app}^i \zz^{{\mu_1}} \otimes \nN_{app}^{p-i}\zz^{{\mu_2}}
\\&=&
                                \sum_{i=0}^{p} 2^{p}\frac{\binom{p}{i}}{\binom{{m_1}}{i}\binom{{m_2}}{p-i}}\frac{(-1)^i}{i!(p-i)!}  \frac{{\mu_1}!}{({\mu_1}-i)!} \frac{{\mu_2}!}{({\mu_2}-p+i)!}y^{p} \zz^{{\mu_2}+{\mu_1}-p} 
\\&\equiv&
{2}^{\,p}\frac {{{\mu_1}+{\mu_2}-\rho\choose \rho}{{\mu_1}\choose \rho}{{\mu_2}\choose \rho}}{{{m_1}\choose p}{{m_2}\choose p}{p \choose \rho}}{\zz}^{{\mu_1}+{\mu_2}-p}{\del}^{\rho}.
\eas

We recall the inversion formula from Theorem \ref{thm:RCGC}:
\bas
\nv^{(i)}_m\otimes \nw^{(j)}_n&=&\sum_{p+k=i+j}\begin{pmatrix} m&n&m+n-2p\\i&j&k\end{pmatrix}\frac{ \binom{m}{i}\binom{n}{j}\binom{m}{p}\binom{n}{p}}{\binom{m+n-2p}{k}\binom{m+n-p+1}{p}}\Join_{m+n-2p}^{(k)}\nv_m\otimes \nw_n.
\eas
This gives us a formula for \(\nN^{l_1} \zz^{\mu_1} \cdot\nN^{l_2} \zz^{\mu_2}\), with \(\nv_{m_i}=\zz^{\mu_i}\) for \(i=1,2\):
\bas
&&\frac{1}{{l_1}!{l_2}!} \nN^{l_1} \zz^{\mu_1} \cdot \nN^{l_2} \zz^{\mu_2}
\\&=&\pi_\circ\sum_{p+k={l_1}+{l_2}}\begin{pmatrix} {m_1}&{m_2}&{m_1}+{m_2}-2p\\{l_1}&{l_2}&k\end{pmatrix}\frac{ \binom{{m_1}}{{l_1}}\binom{{m_2}}{{l_2}}\binom{{m_1}}{p}\binom{{m_2}}{p}}{\binom{{m_1}+{m_2}-2p}{k}\binom{{m_1}+{m_2}-p+1}{p}}\Join_{{m_1}+{m_2}-2p}^{(k)}\nv_{{m_1}}\otimes \nv_{{m_2}}
\\&=&\sum_{p+k={l_1}+{l_2}}\begin{pmatrix} {m_1}&{m_2}&{m_1}+{m_2}-2p\\{l_1}&{l_2}&k\end{pmatrix}
\frac{ \binom{{m_1}}{{l_1}}\binom{{m_2}}{{l_2}}}{\binom{{m_1}+{m_2}-2p}{k}\binom{{m_1}+{m_2}-p+1}{2\rho}\binom{p}{\rho}}
\\&\times&{2}^{p}{\mu_1+\mu_2-\rho\choose \rho}{\mu_1\choose \rho}{\mu_2\choose \rho}\frac{1}{k!} \nN^k \zz^{{\mu_1}+{\mu_2}-p}\del^{\rho}.
\eas
Using Definition \ref{def:lambda} we have
%
\ba\label{eq:nn}
\nN^{l_1} \zz^{\mu_1} \cdot \nN^{l_2} \zz^{\mu_2}&=& \sum_{p+k={l_1}+{l_2}} \lambda^{l_1,\mu_1}_{l_2,\mu_2}(\rho) \nN^k \zz^{{\mu_1}+{\mu_2}-p}\del^{\rho},
\ea
and this is the desired result.
\end{proof}
\begin{rem}\label{rem:lambda0}
	This reduces, using Remark \ref{rem:pis0}, for \(\rho=0\) to
	\bas
\lambda^{l_1,\mu_1}_{l_2,\mu_2}(0)&=&\begin{pmatrix} {m_1}&{m_2}&{m_1}+{m_2}\\{l_1}&{l_2}&l_1+l_2\end{pmatrix}\frac{ \binom{{m_1}}{{l_1}}\binom{{m_2}}{{l_2}}}{\binom{{m_1}+{m_2}}{l_1+l_2}\binom{l_1+l_2}{l_1}}
=
\frac{ \binom{{m_1}}{{l_1}}\binom{{m_2}}{{l_2}}}{\binom{{m_1}+{m_2}}{l_1+	l_2}\binom{l_1+l_2}{l_1}}
=\frac{\binom{m_1+m_2-l_1-l_2}{m_1-l_1}}{\binom{m_1+m_2}{m_1}}.
\eas
\end{rem}
\section{Structure constants for the \(\mathscr{A}_2\)-family}\label{sec:strucc3}
	\begin{defn}
\def\nyA{\mathsf{A}}
The \(3\)D Euler family \(\mathscr{A}_2\)  is defined by the following (affine) vector spaces
	\ba\label{CFamily}
	\mathscr{A}_2= \langle {\nyA}^{l}_{\mu,k}\rangle_{ 0\leqslant l \leqslant 2\mu, \mu, k\in\mathbb{N}_0},\quad
	\widehat{\mathscr{A}}_2= \nN+\mathscr{A}_2,
	\ea 
where \(\nA_{\mu,k}^l=\sad_\nN^l \zz^\mu \del^k \E= (\nN^l \zz^\mu)\del^k \E\).
	We observe that \(\nH \nA_{\mu,k}^l = (m-2l) \nA_{\mu,k}^l\), with \(m=2\mu\).
\end{defn}
The aim of this part, is to derive explicit formulas
for the structure constants of the \(\mathscr{A}_2\)-family.
\begin{thm}\label{thm:FSTC3}
	The following holds, with \(p=2\rho\) :
	\ba\label{eq:3d}
	{[}\nA_{\mu_1,k_1}^{l_1},\nA_{\mu_2,k_2}^{l_2}{]}&=&
	 (\mu_2+2k_2-\mu_1-2k_1) \sum_{p+k={l_1}+{l_2}}  \lambda^{l_1,\mu_1}_{l_2,\mu_2}(\rho) \nA^k_
	{{\mu_1}+{\mu_2}-p,{k_1+k_2+\rho}}. 
	\ea
\end{thm}
\begin{proof}
		The structure constants for \(\mathscr{A}_2\) of triple-zero singularities  are  (on a basis of homogeneous vector fields)
		\bas \label{e}
	[f{\E},g{\E}]=f({\E}g){\E}-g({\E}f){\E}=
	\left(\delta_0\left(g\right)-\delta_0\left(f\right)\right) fg{\E},
	\eas
	where \(\delta_0\) is the degree. This shows that \(\mathscr{A}_2\) is a Lie subalgebra of the \(3\)D polynomial vector fields.
Applying  this  formula  and \eqref{eq:nn} we have
\bas
{[}\nA_{i_1,k_1}^{l_1},\nA_{i_2,k_2}^{l_2}{]}&=&\left(\delta_0(\nA_{i_2,k_2}^{l_2})-\delta_0(\nA_{i_1,k_1}^{l_1}) \right) \nN^{l_1}\zz^{i_1}\cdot\nN^{l_2} \zz^{i_2} \del^{k_1+k_2} \E
\nonumber
\\&=& \left(\delta_0(\nA_{i_2,k_2}^{l_2})-\delta_0(\nA_{i_1,k_1}^{l_1}) \right) \sum_{p+k={l_1}+{l_2}}  \lambda^{l_1,i_1}_{l_2,i_2}(\rho) \nN^k \zz^{{i_1}+{i_2}-p}\del^{k_1+k_2+\rho}\E
\nonumber\label{equ:st3d}
\\&=& \left(\delta_0(\nA_{i_2,k_2}^{l_2})-\delta_0(\nA_{i_1,k_1}^{l_1}) \right) \sum_{p+k={l_1}+{l_2}}  \lambda^{l_1,i_1}_{l_2,i_2}(\rho) \nA^k_
{{i_1}+{i_2}-p,{k_1+k_2+\rho}}. 
\eas
This gives us an equivariant decomposition with explicit structure constants (where explicit is a bit restricted by the \(3j\)-symbols, which are given as a finite sum).
\end{proof}
\begin{cor}
	Notice that if we choose to work \(\mod \del^{k_1+k_2+1}\),  formula \eqref{eq:nn} reduces to
\ba\label{COR:MODDELTA}
{[}\nA_{i_1,k_1}^{l_1},\nA_{i_2,k_2}^{l_2}{]}=&
\left(\delta_0(\nA_{i_2,k_2}^{l_2})-\delta_0(\nA_{i_1,k_1}^{l_1}) \right) \lambda^{l_1,i_1}_{l_2,i_2}(0) \nA^{l_1+l_2}_{{i_1}+{i_2},{k_1+k_2}},& \mod  \del^{k_1+k_2+1}.
\ea
This makes the formula explicit, cf. {\em Remark {\ref{rem:lambda0}}}.
\end{cor}
\begin{exm} 
	These are some examples of structure constants of the \(\mathscr{A}_2\) Lie algebra:
	\bas
	{[{\mathsf A}^{2}_{3,0},{\mathsf A}^{14}_{13,0}]}&=& {\frac {325}{16182}}\nA_{{16,16,0}}-{\frac {208}{93}}\nA_{{14,14,1}}+{
		\frac {146578432}{23}}\nA_{{10,10,3}}-{\frac {7192640}{
			2001}}\nA_{{12,12,2}}
	,
	\\
	{[{\mathsf A}^{8}_{7,2},{\mathsf A}^{5}_{7,1}]}&=&-{\frac {4004}{200583}}\nA^{11}_{{12,4}}
	-{\frac {4330871193600}{46189}}\nA^3_{{4,8}}-{\frac {53760000}{5681}}\nA^7_{{8,6}}+{\frac {5268480000
			}{3553}}\nA^5_{{6,7}}
		\\&&
		+{\frac {312134860800}{221}}\nA^1_{{2,9}}+{\frac 
		{13505184}{482885}}\nA^9_{{10,5}}-{\frac {1001}{4011660
	}}\nA^{13}_{{14,3}}.
		\eas
\end{exm}

		\section{Normal form studies for the \(\widehat{\mathscr{A}}\)-family}
		{
				In this part, we intend
				 to find the structure constants and unique normal form for the \(\mathscr{A}\)  family  of  \(2\)D and \(3\)D nilpotent  singularities,  cf. \cite{baider1992further} and  \cite[Equation 4.19]{JF2019vector}, were in the \(3\)D case  we  apply the inversion formula of the Clebsch-Gordan coefficients. Here, these families are called \(\widehat{\mathscr{A}}_1\) and  \(\widehat{\mathscr{A}}_2\) respectively, see \eqref{eq:A1Family} and \eqref{CFamily}. 
				Apart from technical complications in the \(3\)D case, the unique normal form computation is very similar in these two cases, so it makes sense to read the \(2\)D case before going to the \(3\)D case.

\subsection{ Unique normal for the \(\widehat{\mathscr{A}}_1\)-family}\label{2DBT}
One should note that the unicity  results of   this part  are new for the third level normal form and do not follow from  the results of  \cite{baider1992further}. 

The purpose of this part is to study the unique normal form of the following system 
	\ba\label{2N}
\begin{pmatrix} \dot{x}\\ \noalign{\medskip}\dot{y}\end{pmatrix}
=\begin{pmatrix}0& 0\\ \noalign{\medskip}
	1&0\end{pmatrix} \begin{pmatrix} x\\ \noalign{\medskip}y\end{pmatrix}
+ \sum_{m=0}^{\infty}\sum_{n=0}^{m}  a_{m,n}^{(0)} x^n y^{m-n}\begin{pmatrix} x\\ \noalign{\medskip}y\end{pmatrix},
\ea
with \(a_{0,0}^{(0)}=0\) and \( a_{m,n}^{(0)}\)  real constants.
Let \({\mf {gl}}_2\) is defined by \(\nN,\mathsf{H},\nM, \E\) by the following formulas
\bas
\nN&=&
x\frac{\partial}{\partial y},
\qquad
\mathsf{H}=
y\frac{\partial}{\partial y}-x\frac{\partial}{\partial x},
\\
\nM &=&
y\frac{\partial}{\partial x},
\,\,\,\,\qquad
\E=x\frac{\partial}{\partial x}+y\frac{\partial}{\partial y}. 
\eas
Then let \(\zz\) be the linear form in \(x,y\) that is in \(\ker\nM\) and has \(\nH\)-eigenvalue \(1\).
Recall from \cite{baider1992further,mokhtari2019versal} the \(\Sl\)-representation for two dimensional vector field by 
\ba
{\nA}_{m}^{n}&:=& \nN^n \zz^{m}\E,\qquad\qquad\qquad\; (0\leq n\leq m),
\ea

The following special cases are useful for the unique normal form study of \eqref{2N}:
\ba
\nonumber
{[\nM, \nA^0_m]}&=&0,
\\
\nonumber
{[\nN, \nA^n_m]}&=&\nA^{n+1}_m,n<m,
\\\nonumber
{[\nN, \nA^m_m]}&=&0,
\\\label{S23}
{[\nA^{0}_{\nu}, \nA^n_m]}&=&(m-\nu)\nA^n_{m+\nu}.
\ea
Define
\ba\label{eq:A1Family}
\mathscr{A}_1= \langle {\nA}^{l}_{i} \rangle_{i=1,\cdots,\infty, l=0,\cdots,i },\quad
\widehat{\mathscr{A}}_1=\nN+\mathscr{A}_1. 
\ea
Using the \(\Sl\)-structure  one can rewrite \eqref{2N} in the following form 
\ba\label{eq:2NS}
\nv^{(0)}=\nN+\sum_{s=1}^{\infty}\sum_{l=0}^s {\na}^{(0)}_{s,l} {\nA}^l_s\, \in \widehat{\mathscr{A}}_1,
\ea
in which the   \( {\na}^{(0)}_{s,l}\) are real  constants. 

\begin{prop}\label{prop:firstlevel}
The first level normal form of \eqref{eq:2NS} is given by 
\ba\label{E21}
\nv^{(1)}=\nN+\sum_{s=1}^{\infty} {\na}^{(1)}_s {\nA}^0_s \in \widehat{\mathscr{A}}_1,
\ea
where the \({\na}^{(1)}_s\) are the coefficients of the first level formal normal form. 
\end{prop}
\bpr
The result follows from  the  \(\Sl\)-style normal form applied to  Equation \eqref{eq:2NS}. 
See Proposition \ref{prop:firstlevel} in this paper, \cite{baider1991unique,sanders2003normal} and \cite[ page 201]{SVM2007} for more details.
The transformation generator at grade \(s\) can be obtained by solving the Homological Equation
\[
\sad_\nM\left({\nA}^l_s-\sad_\nN{\nT}^{l}_s\right)=0.
	\]
	One sees that if \(l>0\), \({\nT}^{l}_s={\nA}^{l-1}_s\) is a solution.
The terms \({\nA}^{0}_s\) are in \(\ker\sad_\nM\) and therefore in the chosen complement of \(\im\sad_\nM\) and we can take \({\nT}^{l}_0=0\).
The ease of this solution procedure is the motivation for the choice of basis.
\epr
We finish this part by giving  the  unique normal form of \eqref{eq:2NS}.

Define 
\(
\nu_{1}:=\min\{s\mid {\na}^{(1)}_s \neq 0\},
\) where the \({\na}^{(1)}_s\) are  the first level normal form coefficients of \eqref{E21}.
\begin{rem}
Since we are going to divide by \({\na}^{(1)}_{\nu_{1}}\)  in the process of computing the transformation, one might want to impose some more requirements on this number. One could for instance 
stipulate that \(|{\na}^1_{\nu_{1}}|>0.3\), thus avoiding  big numbers in the formal transformation. The numbers that we consider too small 
could be  seen as versal deformation parameters.
In this paper, we will not complicate our results by these considerations but leave it to the reader to fill in the necessary modifications.    
\end{rem}
By rewriting Equation  \eqref{E21} we obtain
\ba\label{eq:nu1}
\nv^{(1)}:=\nN+a_{\nu_{1}}^{(1)}\nA^0_{\nu_{1}}+\sum_{s=\nu_{1}+1}^{\infty} a^{(1)}_{s} \nA^0_s.
\ea
Define the grading  \(\delta_{2}\) by 
\bas
\delta_{2}(\nZ^l_k):=2(k+\nu_{1} l).
\eas
Hence  \(\delta_{2}(\nA^0_{\nu_{1}})=\delta_{2}(\nN)=2\nu_{1}.\)  Define \(\nX^{2\nu_{1}}=\nN+a^{(1)}_{\nu_{1}}\nA^0_{\nu_{1}}.\)
Now we use the  kernel of \(\nN\)  which is  generated  by  \( \lbrace \nA^m_m \mid m \geq 1 \rbrace\) to eliminate  extra  terms from \eqref{eq:nu1}. But first we  formulate the following lemma.
\begin{lem}\label{lem:GT2D}
	There exists a transformation \(\nT^n_m=\sum_{i=0}^{n-1} \alpha_i \nA^{n-i-1}_{m+i\nu_{1}}\in\mathscr{A}_{1}\)   such that the following holds: 
	\bas
	[\nX^{2\nu_{1}},\nT^n_m]+\nA^n_m=\alpha_{n} \nA^{0}_{m+n \nu_{1}},
	\eas
	in which 
	\bas
\alpha_i=-\prod_{j=1}^{i}a^{(1)}_{\nu_{1}}  \left(\delta_0(\nA^0_{\nu_{1}})-\delta_0( \nA^{n-j}_{m+(j-1)\nu_{1}})\right),\,\,\ \hbox{for}\,\,\, 0\leq i\leq n. 
	\eas
\end{lem}
\bpr
We intend to find the coefficient of previous vector field such that  \bas	[\nT^n_m, \nX^{2\nu_{1}}]+\nA^n_m\in \ker(\sad_\nM).\eas
 By straightforward calculation one has 
\bas
[ \nX^{2\nu_{1}},\nT^n_m]+\nA^n_m&=&\nA^n_m+[\nN+a_{\nu_{1}}^{(1)}\nA^0_{\nu_{1}},\sum_{i=0}^{n-1} \alpha_i \nA^{n-i-1}_{m+i\nu_{1}}]
\\&=&\nA^n_m+\sum_{i=0}^{n-1} \alpha_i \nA^{n-i}_{m+i\nu_{1}}+\sum_{i=0}^{n-1}
 \left(\delta_0( \nA^{n-i-1}_{m+i\nu_{1}})-\delta_0(\nA^0_{\nu_{1}})\right)
 \alpha_i a_{\nu_{1}}^{(1)} \nA^{n-i-1}_{m+i\nu_{1}+\nu_1}
\\&=&\nA^n_m+ \alpha_0 \nA^{n}_{m}+\sum_{i=1}^{n-1} \left( \alpha_i+
 \left(\delta_0( \nA^{n-i}_{m+(i-1)\nu_{1}})-\delta_0(\nA^0_{\nu_{1}})\right) 
 \alpha_{i-1} a_{\nu_{1}}^{(1)} \right)\nA^{n-i}_{m+i\nu_{1}}
\\&&+ \left(\delta_0( \nA^{0}_{m+(n-1)\nu_{1}})-\delta_0(\nA^0_{\nu_{1}})\right) \alpha_{n-1} a_{\nu_{1}}^{(1)} \nA^{0}_{m+n\nu_{1}}.
\eas
Now, let
\(
\alpha_0=-1
\) 
and 
\[
\alpha_i=
 \left(\delta_0(\nA^0_{\nu_{1}})-\delta_0( \nA^{n-i}_{m+(i-1)\nu_{1}})\right)a_{\nu_{1}}^{(1)} 
\alpha_{i-1}  ,\,\,\ \hbox{for}\,\,\, 1\leq i\leq n.
\]
Finally we obtain 
\[[ \nX^{2\nu_{1}},\nT^n_m]+\nA^n_m=  \left(\delta_0( \nA^{0}_{m+(n-1)\nu_{1}})-\delta_0(\nA^0_{\nu_{1}})\right)\alpha_{n-1}a_{\nu_{1}}^{(1)} \nA^{0}_{m+n \nu_{1}}= -\alpha_{n} \nA^{0}_{m+n \nu_{1}}.\]
\epr

\begin{thm}
	There exists a  transformation that takes \eqref{2N} into  its  second level  normal form  
\ba\label{22N}
\nv^{(2)}:=\nN+a_{\nu_{1}}^{(1)} \nA^0_{\nu_{1}}+\sum_{s=\nu_{1}+1}^{\infty} a^{(2)}_{s} \nA^0_s,
\ea
where  \(a^{(2)}_s\) is zero for \(s\stackrel{(1+\nu_{1})}{\equiv} \nu_{1}.\)
\end{thm}
\bpr
As mentioned before,  \(\nA^m_m\) is in   \(\ker\sad_\nN\).
Then, 
\bas
[\nX^{2\nu_{1}},\nA^m_m]= (m-\nu_{1}) \nA^m_{m+\nu_{1}}.
\eas
Let us assume that  \(m\neq \nu_{1}.\) If \(m=\nu_{1}\)  then \(A^{\nu_{1}}_{\nu_{1}}\) generates the kernel for \(\nX^{2\nu_{1}}.\)
This kernel term will be used  in the last theorem of this section to provide the unique normal form.   
Now, by applying Lemma \ref{lem:GT2D}  we get 
\bas
[\nX^{2\nu_{1}},\nA^m_m+(m-\nu_{1})\nT^m_{m+\nu_{1}}]&=& -\alpha_{m}(m-\nu_{1}) \nA^{0}_{m+m \nu_{1}+\nu_{1}}, 
\eas
where \(\alpha_m\) is nonzero.
\epr
From \eqref{S23}  we have \([\nA^{0}_{\nu_{1}}, \nA^{\nu_{1}}_{\nu_{1}}]=0.\) Then \(\nA^{\nu_{1}}_{\nu_{1}}\) generates the kernel for
\(\nX^{2\nu_{1}}.\)
Set
 \[\nu_{2}=\min\{ s\mid a^{(2)}_s\neq 0\}.\]
 We close  this section by  giving the main theorem about the unique normal form of \eqref{eq:2NS}.
\begin{thm}
	There exists the  transformation that takes \eqref{22N} into  its  unique (third level) normal form  
	\ba\label{u2N}
	\nv^{(3)}:=\nN+a_{\nu_{1}}^{(1)}\nA^0_{\nu_{1}}+a_{\nu_{2}}^{(2)} \nA^0_{\nu_{2}} +\sum_{s=\nu_{2}+1}^{\infty} a^{(3)}_{s} \nA^0_s,
	\ea
	or equivalently
	\ba\label{u2N}
	\begin{pmatrix} \dot{x}\\ \noalign{\medskip}\dot{y}\end{pmatrix}
	=\begin{pmatrix}0& 0\\ \noalign{\medskip}
		1&0\end{pmatrix} \begin{pmatrix} x\\ \noalign{\medskip}y\end{pmatrix}
	+a_{\nu_{1}}^{(1)} y^{\nu_{1}}  \begin{pmatrix} x\\ \noalign{\medskip}y\end{pmatrix}
	+	a_{\nu_{2}}^{(2)} y^{\nu_{2}}  \begin{pmatrix} x\\ \noalign{\medskip}y\end{pmatrix}+ \sum_{s=\nu_{2}+1}^{\infty}  a_{s}^{(3)}  y^{s}\begin{pmatrix} x\\ \noalign{\medskip}y\end{pmatrix},
	\ea 
	where  \(a^{(3)}_s\) is zero for \(s\stackrel{(1+\nu_{1})}{\equiv} \nu_{1}\) and  \(s\stackrel{(1+\nu_{1})}{\equiv} \nu_{2}.\)
\end{thm}
\bpr
From the structure constants one has 
\bas
[\nA^0_{\mu},\nA^m_m]=(m-\mu)A^m_{m+\mu},\quad \hbox{with }\left\{ \begin{array}{lcr}\mu=\nu_{2},
	 &\hbox{ if } &m=\nu_{1},\\\mu=\nu_{1}, &\hbox{ if }& m\neq \nu_{1}.\end{array} \right.
\eas

By applying Lemma \ref{lem:GT2D} one obtains 
\bas
[\nA^0_{\mu},\nA^m_m]+[\nX^{2\nu_{1}},(m-\mu)\nT^m_{m+\mu}]=-
(m-\mu)\alpha_{m} \nA^{0}_{m+\mu+m \nu_{1}},
\eas
which is non-zero. Hence for \(m=\nu_{1}\) and \(\mu=\nu_{2}\) one obtains
\bas
[\nA^0_{\nu_{2}},\nA^{\nu_{1}}_{\nu_{1}}]
+[\nX^{2\nu_{1}},(\nu_{1}-\nu_{2})\nT^{\nu_{1}}_{\nu_{1}+\nu_{2}}]&=&
(\nu_{2}-\nu_{1}) \alpha_{m}
 \nA^{0}_{\nu_{1}+\nu_{2}+\nu_{1}^2 }.
\eas
Therefore,   the term  \(\nA^{0}_{\nu_{1}+\nu_{2}+\nu_{1}^2}\) can be eliminated from \eqref{22N}.
\epr
}


\subsection{ Unique normal for the \(\widehat{\mathscr{A}}_2\)-family }\label{3DBT}
In this section, we study 
 the unique normal form for the class  of equations of the form 
\ba\label{3N}
\begin{pmatrix} \dot{x}\\ \noalign{\medskip}\dot{y}\\ \noalign{\medskip}\dot{z}\end{pmatrix}
=\begin{pmatrix}0& 0&0\\ \noalign{\medskip}
	1&0&0
\\ \noalign{\medskip}
0&2&0
\end{pmatrix} \begin{pmatrix} x\\ \noalign{\medskip}y \\ \noalign{\medskip}{z}\end{pmatrix}
+\sum_{k=0}^{\infty} \sum_{\mu=0}^{\infty} \sum_{l=0}^{2\mu} 
{ a}^{(0)}_{l,\mu,k}{\mathsf a}^l_{\mu,k}\begin{pmatrix} x\\ \noalign{\medskip}y\\ \noalign{\medskip}{z}\end{pmatrix},
\ea
where 
\(
{ a}^{(0)}_{l,\mu,k}\in{\mathbb{R}},\) \(  { a}^{(0)}_{0,0,0}=0\) and \({\mathsf a}^l_{\mu,k}={\del^k } \nN^l{\zz}^\mu\)  as defined in Section \ref{sec:strucc3}
or 
\ba\label{3Nv}
\nv^{(0)}
= \mathsf{N}
+\sum_{k=0}^{\infty} \sum_{\mu=0}^{\infty} \sum_{l=0}^{2\mu} 
{ a}^{(0)}_{l,\mu,k}{\mathsf A}^l_{\mu,k}\in\widehat{\mathscr{A}}_2,
\ea
where \({\mathsf A}^l_{\mu,k}={\mathsf a}^l_{\mu,k}\mathsf{E}\).
\begin{lem}
	There exists a sequence of   vector fields belonging  to \({\mathscr{A}_2}\)
	which takes \eqref{3N} to the following first level normal form
	\ba\label{Eq:FL3}
	\nv^{(1)}=\nN+\sum_{\mu=0}^{\infty}\sum_{k=0}^{\infty} a_{\mu,k}^{(1)}{
		 \nA}^0_{\mu,k},
	\ea
	where \(a_{\mu,k}^{(1)}\) denotes  the coefficients of vector fields in the first level normal form.
\end{lem}
	\bpr
	Recall that \(\ker\sad_\nM\) is  \(\Sl\)-style normal form. Then, 
	by fixing  this style 
the proof follows from 
\(
	\sad_{\nM} {\nA}^{0}_{\mu,k}= 0
	.\) 
	\epr
Now, define  the grading function by taking \(r=s+r_s\) and
\bas
{\delta}_{r_s,s}({\nZ}^l_{i,k}):=
(r_s+1)(i+2k)+(l-i)(r_s+2s),
\eas
in which \(r_s=\min\{m\mid a^{(1)}_{m,s}\neq 0\}\).
Hence, 
\bas
{\delta}_{r_s,s}({\nN})={\delta}_{r_s,s}({\nA}^0_{r_s,s})=r_s+2s.
\eas
 Remark that \(\delta_{0,0}({\nZ}^l_{i,k})=i+2k\) coincides with the previous definition of \(\delta_0\).
 Define, 
 \bas
 \nX^{r_s+2s}&=&{\nN}+a^{(1)}_{r_s,s}{\nA}^0_{r_s,s}.
 \eas
 Before we start the study of further elimination of  terms from Equation \eqref{Eq:FL3} we would like to discuss in details about the filtering approach that plays the main role in the rest of this paper. 
 
 From Theorem \eqref{equ:st3d}, it follows that  the power of \(\del\) measures the filtering. 
 In the process  of normalization we would be able to get rid of all terms (with power of \(\del\))  that generated 
 in the next level of normal form.
 
 Assume that in the second level normal form one wants to eliminate the terms \(\nA^0_{i,s}\) where \(i>r_s,\)
 then one does not need to take into account  the terms with a higher power of \(\del^s.\)
 Therefore, the transformation \(\nT^{i-r_s}\) can be obtained by solving the following  
 \bas
 [\nX^{r_s+2s}, \nT^{i-r_s}]\equiv\nA^0_{i,s},
 \eas 
 where we start, fixing the degree \(\delta=i+2s\), our computation with \(s=0\). With the degree fixed,
 the filtered computation will take only a finite number of steps. So first we eliminate \(A_{\delta,0}^0\),
 then \(A_{\delta-2,1}^0\), until we have exhausted all terms in \(\ker\sad_\nM\) with \(\delta_{r_s,s}=\delta\).
 \begin{thm} \label{thm:3dtransformation}
 	There exists transformation as \(\nT^n_{m,k}\)
 	such that the following holds:
 	\bas
 	[	\nT^n_{m,k},{\nX}^{r_s+2s}]+\nA^n_{m,k}&\equiv&-\alpha _{n}\nA^{0}_{m+nr_s,k+ns}.
 	\eas 
 Here we define \(	\alpha_i\) to be
 \bas
 \alpha _i
 &=&
 -\prod_{j=1}^{i} a^{(1)}_{r_s,s}
 \left(\delta_0({\nA}^0_{r_s,s})-\delta_0(\nA^{n-j+1}_{m+jr_s-r_s,k+js-s})\right)\lambda^{0,r_s}_{n-j+1,m+jr_s-r_s}(0),\,\,\ \hbox{for}\,\,\, 0\leq i\leq n.
 \eas
 \end{thm}
\bpr
Define 
\bas
	\nT^n_{m,k}=\sum_{i=0}^{n-1} \alpha_i\nA^{n-i-1}_{m+ir_s,k+is}.
\eas
Then, 
\bas
[{\nX}^{r_s+2s},\nT^n_{m,k}]+\nA^n_{m,k}&=&\nA^n_{m,k}+\sum_{i=0}^{n-1} \alpha_i[{\nN},\nA^{n-i-1}_{m+ir_s,k+is}]+\sum_{i=0}^{n-1} \alpha_i[a_{r_s,s}^{(1)}{\nA}^0_{r_s,s},\nA^{n-i-1}_{m+ir_s,k+is}]
\\
&\equiv&\nA^n_{m,k}+\sum_{i=0}^{n-1} \alpha_i\nA^{n-i}_{m+ir_s,k+is}
\\&&-
\sum_{i=0}^{n-1} \alpha_i a^{(1)}_{r_s,s}
\left(\delta_0({\nA}^0_{r_s,s})-\delta_0(\nA^{n-i-1}_{m+ir_s,k+is})\right)\lambda^{0,r_s}_{n-i,m+ir_s}(0)\nA^{n-i-1}_{m+ir_s+r_s,k+is+s}
\\
&=&\nA^n_{m,k}+\alpha_0 \nA^n_{m,k}+\sum_{i=1}^{n-1} \alpha_i\nA^{n-i}_{m+ir_s,k+is}
\\&&
-\sum_{i=1}^{n-1} \alpha_{i-1} a^{(1)}_{r_s,s}
\left(\delta_0({\nA}^0_{r_s,s})-\delta_0(\nA^{n-i}_{m+ir_s-r_s,k+is-s})\right)\lambda^{0,r_s}_{n-i+1,m+ir_s-r_s}(0)\nA^{n-i}_{m+ir_s,k+is}
\\&&
-\alpha_{n-1} a^{(1)}_{r_s,s}
\left(\delta_0({\nA}^0_{r_s,s})-\delta_0(\nA^{0}_{m+nr_s-r_s,k+ns-s})\right)\lambda^{0,r_s}_{1,m+nr_s-r_s}(0)\nA^{0}_{m+nr_s,k+ns}.
\eas
Now, let 
\(\alpha _0=-1\) and 
\bas \alpha _i
-
\alpha _{i-1}a^{(1)}_{r_s,s}
\left(\delta_0({\nA}^0_{r_s,s})-\delta_0(\nA^{n-i+1}_{m+ir_s-r_s,k+is-s})\right)\lambda^{0,r_s}_{n-i+1,m+ir_s-r_s}(0)=0,\,\,\ \hbox{for}\,\,\, 1\leq i\leq n-1.
\eas
Hence, we find 
\bas
 \alpha _i
&=&
-\prod_{j=1}^{i} a^{(1)}_{r_s,s}
	\left(\delta_0({\nA}^0_{r_s,s})-\delta_0(\nA^{n-j+1}_{m+jr_s-r_s,k+js-s})\right)\lambda^{0,r_s}_{n-j+1,m+jr_s-r_s}(0),\,\,\ \hbox{for}\,\,\, 0\leq i\leq n-1.
\eas
Therefore,
\bas
[ \nT^n_{m,k},{\nX}^{r_s+2s}]+\nA^n_{m,k}
&=&-\alpha _{n}\nA^{0}_{m+nr_s,k+ns}.
\eas 
This is exactly the statement in the theorem.
\epr

\begin{thm}\label{secondNF3}
	There exists  a sequence of   invertible transformation that sends  \eqref{Eq:FL3}  into the  second level  normal form 
\ba\label{Eq:23N}
\nv^{(2)}:=\nN+a^{(1)}_{r_{s},s}\nA^0_{r_{s},s}+\sum_{p=r_{s}+1}^{\infty} \sum_{q=s+1}^{\infty} a^{(2)}_{p,q} {
	\nA}^0_{p,q},
\ea
in which \(a^{(2)}_{p,q}=0\)  where   \(p\stackrel{(1+2r_s)}{\equiv} r_s,\)   \(q=k+(1+2m)s\)  and 
\(m+2k\neq r_s+2s,\) for natural numbers \(m.\)
\end{thm}
\bpr
From \bas
{[{\nN}, {\nA}^{l}_{i,k}]}&=& {\nA}^{l+1}_{i,k},
\eas
we observe that, the  null  space of  \(\nN\)  is generated by  
\({	\nA}^{2m}_{m,k}\) for all natural numbers  \(m,k.\) 
Therefore,  we  use 
this kernel  term  to remove an  extra term
from  \eqref{Eq:FL3}:
 \bas
 [{	\nA}^{2m}_{m,k},\nX^{r_s+2s}]\equiv \left(\delta_0(\nA^0_{r_s,s})-\delta_0(\nA^{2m}_{m,k})\right)\lambda^{2m,m}_{0,r_s}(0){	\nA}^{2m}_{m+r_s,k+s}.
 \eas
 Assume that 
  \(\delta_0(\nA^0_{r_s,s})\neq\delta_0(\nA^{2m}_{m,k}).\) Using 
Theorem \eqref{thm:3dtransformation} we  know that  there exists a  transformation as  
\(\nT^{2m}_{m+r_s,k+s}\)  such that 
\bas
&&[{	\nA}^{2m}_{m,k}+\left(\delta_0(\nA^0_{r_s,s})-\delta_0(\nA^{2m}_{m,k})\right)\lambda^{2m,m}_{0,r_s}(0)\nT^{2m}_{m+r_s,k+s},\nX^{r_s+2s}]
\\
&&=-\left(\delta_0(\nA^0_{r_s,s})-\delta_0(\nA^{2m}_{m,k})\right)\lambda^{2m,m}_{0,r_s}(0)
\alpha_{2m}\nA^{0}_{m+r_s+2mr_s,k+s+2ms},
\eas
which is non-zero. Hence, the proof follows.
\epr
As  it shows for \(\delta_0(\nA^0_{r_s,s})=\delta_0(\nA^{2m}_{m,k})\) which 
reads \(m_1=m=r_s+2(s-k)\)
the term   \(\nA^{2m_1}_{m_1,k}\) is kernel of \(\nX^{r_s+2s}.\) We shall use this kernel for further possible reduction of terms form        the \eqref{Eq:23N}. 
Rewrite  \eqref{Eq:23N} by
\bas
\nv^{(2)}:=\nN+a^{(1)}_{r_{s},s}\nA^0_{r_{s},s}+a^{(2)}_{r_{s_1},s_1}\nA^0_{r_{s_1},s_1}+\sum_{p=r_{s_1}+1}^{\infty} \sum_{q=s_1+1}^{\infty} a^{(2)}_{p,q} {
	\nA}^0_{p,q}.
\eas
\begin{thm}
	The  unique normal form of \eqref{Eq:FL3} is given by
\ba\label{Eq:33N}
\nv^{(3)}:=\nN+a^{(1)}_{r_{s},s}\nA^0_{r_{s},s}
+a^{(2)}_{r_{s_1},s_1}\nA^0_{r_{s_1},s_1}
+\sum _{p=r_{s_1}+1}^{\infty} \sum_{p={s_1}+1}^{\infty} a^{(3)}_{p,q}{
	\nA}^0_{p,q},
\ea
in which \(a^{(3)}_{p,q}=0\)  where   \(p\stackrel{(1+2r_s)}{\equiv} r_s,\)   \(p\stackrel{(1+2r_{s_1})}{\equiv} r_{s_1},\) and 
\(q=k+s_1+2m_1s_1\) where \(m_1\in \mathbb{N}\) and \(m_1=r_s+2(s-k).\)
\end{thm}
\bpr
The only transformation that has not been used so far is \(\nA^{2m_1}_{m_1,k}.\)
Hereby,
\bas
[\nA^{2m_1}_{m_1,k},\nA^0_{r_{s_1},s_1}]\equiv
\left(
\delta_{0}(\nA^0_{r_{s_1},s_1})-\delta_0
(\nA^{2m_1}_{m_1,k})\right) \lambda^{2m_1,m_1}_{0,r_s}(0)\nA^{2m_1}_{{m_1+r_{s_1},k+s_1}}.
\eas
Using 
Theorem \ref{thm:3dtransformation} we have 
\bas
&&[\nA^{2m_1}_{m_1,k},\nA^0_{r_{s_1},s_1}]+[\left(
\delta_{0}(\nA^0_{r_{s_1},s_1})-\delta_0
(\nA^{2m_1}_{m_1,k})\right)\lambda^{2m_1,m_1}_{0,r_s}(0)\nT^{2m_1}_{{m_1+r_s,k+s}},\nX^{r_s+2s}]\equiv
 \\&&-\alpha_{2m_1}\left(
 \delta_{0}(\nA^0_{r_{s_1},s_1})-\delta_0
 (\nA^{2m_1}_{m_1,k})\right)\lambda^{2m_1,m_1}_{0,r_s}(0)\nA^{0}_{m_1+r_{s_1}+2m_1r_s,k+s_1+2m_1s_1}.
\eas

\epr
	\begin{rem}
		One should note that for  given \(\delta=j\) one can have more  than one kernel term. 
		Then the matter is  to figure out whether these  kernel  terms  generate   a common  vector field in the process of finding  the unique normal form.
	Assume that 
		 two elements with grade \(j\) such as \(\nA^{2{m_1}}_{{m_1},{k_1}},\nA^{2{m_2}}_{{m_2},{k_2}}\) where \(m_1\neq m_2\) and \(k_1\neq k_2\) are given. Then,
	\bas
{[	{\nA}^{2m_1}_{m_1,k_1},\nA^0_{r_s,s}]}&=&\left<{\nA}^{n_1}_{{{m_1}+{r_s}-2\,p_1},{k_1+s+p_1}}\right>_{2p_1+n_1={2m_1}},
\\
{[{\nA}^{2m_2}_{m_2,k_2},\nA^0_{r_s,s}]}&=&\left<{\nA}^{n_2}_{{{m_1}+{r_s}-2\,p_2},{k_2+s+p_2}}\right>_{2p_2+n_2={2m_2}}.
\eas
Suppose   that these kernel terms generate the same terms then it means that there exist \(p_1\) and \(p_2\) such that 
\bas
{\nA}^{n_1}_{{{m_1}+{r_s}-2\,p_1},{k_1+s+p_1}}={\nA}^{n_2}_{{{m_2}+{r_s}-2\,p_2},{k_2+s+p_2}},
\eas
and this implies that 
\(
n_1=n_2,\,\,
{{m_1}-2\,p_1}={{m_2}-2\,p_2},\,\hbox{and}\,\,
{k_1+p_1}=k_2+p_2.
\)
On the other  hand we have 
\bas
2p_1+n_1={2m_1},  \hbox{and}\,\,2p_2+n_2={2m_2},
\eas
	since \(n_1=n_2\) then one has \(m_1-m_2=p_1-p_2\) which is  in contradiction with \(m_1-m_2=2(p_1-p_2)\) when \(m_1\neq m_2.\)
			
	\end{rem}

\section{Example from \(\widehat{\mathscr{A}}_2\)}
In this part, the  unique normal form up to fourth order  for a given \(3\)D nilpotent singularities is computed using the results from Sections \ref{sec:strucc3} and \ref{3DBT}.
Consider  the following  dynamical system
\bas
{\dot{x}}&=&\left( a^0_{{1,0}}z+2\,a^1_{{1,0}}y+2\,a^2_{{1,0}}x
\right) x
\\&&
+ \Big( a^0_{{2,0}}{z}^{2}+4\,a^1_{{2,0}}yz-a^0_{{0,1}}{y}^{2}+8\,a^2_{{2,0}}{y}^{2}
+\left(a^0_{{0,1}}+4\,a^2_{{2,0}}\right)xz
+24\,a^3_{{2,0}}x
y+24\,a^4_{{2,0}}{x}^{2} \Big) x
\\&&
+ \Big( a^0_{{3,0}}{z}^{3}+6\,a^1_{{3,0}}y{z}^{2}
+\left(-a^0_{{1,1}}
+24\,a^2_{{3,0}}\right){y}^{2}z+\left(-2\,a^1_{{1,1}}+48\,a^3_{{3,0}}\right){y}^{3}+\left(a^0_{{1,1}}+6\,
a^2_{{3,0}}\right)x{z}^{2}
\\&&
+\left(2\,a^1_{{1,1}}+72
\,a^3_{{3,0}}\right)xyz+\left(-2\,a^2_{{1,1}}
+
288\,a^4_{{3,0}}\right)x{y}^{2}+\left(2\,a^2_{{1,1}}+72\,a^4_{{3,0}}\right){x}^{2}z
\\&&
+720\,a^5_{{3,0}}{x}^{2}y+720\,a^6_{{3,0}}{x}^{3} \Big) x,
\\
{\dot{y}}&=&x+
 \left( a^0_{{1,0}}z+2\,a^1_{{1,0}}y+2\,a^2_{{1,0}}x
\right) y+ \Big( a^0_{{2,0}}{z}^{2}+4\,a^1_{{2,0}}yz-a^0_{{0,1}}{y}^{2}+8\,a^2_{{2,0}}{y}^{2}+a^0_{{0,
		1}}xz
	+4\,a^2_{{2,0}}xz
		\\&&
	+24\,a^3_{{2,0}}xy
+24\,a^4_{{2,0}}{x}^{2} \Big) y
+ \Big( a^0_{{3,0}}{z}^{3}+6\,a^1_{{3,0}}y{z}^{2}-a^0_{{1,1}}{y}^{2}z+24\,a^2_{{3,0}}{y}^{2}z-2\,a^1_{{1,1}}{y}^{3}
+48\,a^3_{{3,0}}{y}^{3}
\\&&
+a^0_{{1,1}}x{z}^{2}
+6\,
a^2_{{3,0}}x{z}^{2}+2\,a^1_{{1,1}}xyz+72
\,a^3_{{3,0}}xyz-2\,a^2_{{1,1}}x{y}^{2}+
288\,a^4_{{3,0}}x{y}^{2}+2\,a^2_{{1,1}}{x}^{2}z
\\&&
+72\,a^4_{{3,0}}{x}^{2}z
+720\,a^5_{{3,0}}{x}^{2}y
	+720\,a^6_{{3,0}}{x}^{3} \Big) y,
\\
{\dot{z}}&=& 2y+
\left( a^0_{{1,0}}z+2\,a^1_{{1,0}}y+2\,a^2_{{1,0}}x
\right) z
\\&&
+ \left( a^0_{{2,0}}{z}^{2}+4\,a^1_{{2,0}}yz-a^0_{{0,1}}{y}^{2}+8\,a^2_{{2,0}}{y}^{2}+a^0_{{0,
		1}}xz+4\,a^2_{{2,0}}xz+24\,a^3_{{2,0}}xy+24\,a^4_{{2,0}}{x}^{2} \right) z
	\\&&
	+ \Big( a^0_{{3,0}}{z}^{3}+6\,a^1_{{3,0}}y{z}^{2}-\left(a^0_{{1,1}}+24\,a^2_{{3,0}}\right){y}^{2}z-2\,a^1_{{1,1}}{y}^{3}
	+48\,a^3_{{3,0}}{y}^{3}+\left(a^0_{{1,1}}+6\,
a^2_{{3,0}}\right)x{z}^{2}
\\&&
+\left(2\,a^1_{{1,1}}+72
\,a^3_{{3,0}}\right)xyz
-\left(2\,a^2_{{1,1}}+
288\,a^4_{{3,0}}\right)x{y}^{2}+\left(2\,a^2_{{1,1}}+72\,a^4_{{3,0}}\right){x}^{2}z
+720\,a^5_{{3,0}}{x}^{2}y+720\,a^6_{{3,0}}{x}^{3} \Big) z,
\eas
or equivalently 
\ba\label{eq:3NS}
v^{(0:0)}_{[4]}&:=&\nN+v_1^{(0:0)}+v_2^{(0:0)}+v_3^{(0:0)},
\ea
where the subindex \([4]\) tells us we do not consider terms of grade \(4\) and higher, and
\bas
v_1^{(0:0)}&=&a^0_{1,0}\nA^0_{1,0}+a^2_{1,0}\nA^2_{1,0}+a^1_{1,0}\nA^1_{1,0},
\\
v^{(0:0)}_2&=&a^0_{0,1}\nA^0_{0,1}
+a^1_{2,0}\nA^1_{2,0}+a^0_{2,0}\nA^0_{2,0}
+a^2_{2,0}\nA^2_{2,0}
+a^3_{2,0}\nA^3_{2,0}+a^4_{2,0}\nA^4_{2,0},
\\
v_3^{(0:0)}&=&
a^0_{1,1}\nA^0_{1,1}+a^0_{3,0}\nA^0_{3,0}+a^1_{1,1}\nA^1_{1,1}+a^1_{3,0}\nA^1_{3,0}
+a^2_{1,1}\nA^2_{1,1}+a^2_{3,0}\nA^2_{3,0}
\nonumber
\\&&
+a^3_{3,0}\nA^3_{3,0}+a^4_{3,0}\nA^4_{3,0}+a^5_{3,0}\nA^5_{3,0}
+a^6_{3,0}\nA^6_{3,0},
\eas
and the coefficients  are  real constants. 
\begin{thm}
	The unique normal form of \eqref{eq:3NS} up to order four is given by
	Equivalently, 
	\bas
	{\dot{x}}&=&a^0_{0,1}zx+a^0_{0,1}(xz-y^2)x+a^0_{2,0}z^2x
	\\&&
	+\left( a^1_{{1,0}}a^0_{{1,1}}+24\,a^1_{{2,0}}a^2_{{3,0}}-6\,a^2_{{1,0}}a^2_{{3,0}}-6\,a^2_{{2,0}}a^2_{{3,0}}+4\,a^3_{{2,0}}a^2_{{3,0}}
	+a^0_{{1,1}} \right)xz(xz-y^2)
	\\&&
	+a^0_{{3,0}} \left( a^1_{{1,0}}+a^1_{{2,0}}+
	1 \right)xz^3,
	\\
	{\dot{y}}&=&x+a^0_{0,1}zy+a^0_{0,1}(xz-y^2)y+a^0_{2,0}z^2y
	\\&&
	+\left( a^1_{{1,0}}a^0_{{1,1}}+24\,a^1_{{2,0}}a^2_{{3,0}}-6\,a^2_{{1,0}}a^2_{{3,0}}-6\,a^2_{{2,0}}a^2_{{3,0}}+4\,a^3_{{2,0}}a^2_{{3,0}}
	+a^0_{{1,1}} \right)yz(xz-y^2)
	\\&&
	a^0_{{3,0}} \left( a^1_{{1,0}}+a^1_{{2,0}}+
	1 \right)yz^3,
	\\
	{\dot{z}}&=&2y+a^0_{0,1}z^2+a^0_{0,1}(xz-y^2)z+a^0_{2,0}z^3
	\\&&
	+\left( a^1_{{1,0}}a^0_{{1,1}}+24\,a^1_{{2,0}}a^2_{{3,0}}-6\,a^2_{{1,0}}a^2_{{3,0}}-6\,a^2_{{2,0}}a^2_{{3,0}}+4\,a^3_{{2,0}}a^2_{{3,0}}
	+a^0_{{1,1}} \right)z^2(xz-y^2),
	\\&&
	+a^0_{{3,0}} \left( a^1_{{1,0}}+a^1_{{2,0}}+
	1 \right)z^4.
	\eas
	And  in terms of invariants \(\zz,\del\)  is as
	\bas
	{\dot{x}}&=&a^0_{0,1}\zz x+a^0_{0,1}\del x+a^0_{2,0}\zz^2 x
	\\&&
	+\left( a^1_{{1,0}}a^0_{{1,1}}+24\,a^1_{{2,0}}a^2_{{3,0}}-6\,a^2_{{1,0}}a^2_{{3,0}}-6\,a^2_{{2,0}}a^2_{{3,0}}+4\,a^3_{{2,0}}a^2_{{3,0}}
	+a^0_{{1,1}} \right)\zz \del x
	\\&&
	+a^0_{{3,0}} \left( a^1_{{1,0}}+a^1_{{2,0}}+
	1 \right)\zz^3 x,
	\\
	{\dot{y}}&=&x+a^0_{0,1}\zz y+a^0_{0,1}\del y+a^0_{2,0}\zz^2 y
	\\&&
	+\left( a^1_{{1,0}}a^0_{{1,1}}+24\,a^1_{{2,0}}a^2_{{3,0}}-6\,a^2_{{1,0}}a^2_{{3,0}}-6\,a^2_{{2,0}}a^2_{{3,0}}+4\,a^3_{{2,0}}a^2_{{3,0}}
	+a^0_{{1,1}} \right)\zz\del y
	\\&&
	a^0_{{3,0}} \left( a^1_{{1,0}}+a^1_{{2,0}}+
	1 \right)\zz^3 y,
	\\
	{\dot{z}}&=&2y+a^0_{0,1}\zz z+a^0_{0,1}\del z+a^0_{2,0}\zz^2 z
	\\&&
	+\left( a^1_{{1,0}}a^0_{{1,1}}+24\,a^1_{{2,0}}a^2_{{3,0}}-6\,a^2_{{1,0}}a^2_{{3,0}}-6\,a^2_{{2,0}}a^2_{{3,0}}+4\,a^3_{{2,0}}a^2_{{3,0}}
	+a^0_{{1,1}} \right)\zz \del z
	\\&&
	+a^0_{{3,0}} \left( a^1_{{1,0}}+a^1_{{2,0}}+
	1 \right)\zz^3 z.
	\eas
	\end{thm}
\bpr
At the first level of normal form  we transform the first and the second order terms  in the normal form simultaneously. Set
\bas
\nT_{1}^1&=&a^1_{1,0}\nA^0_{1,0}+a^2_{1,0}\nA^1_{1,0},
\\
\nT_2^1&=&
a^1_{2,0}\nA^0_{2,0}
+a^2_{2,0}\nA^1_{2,0}
+a^3_{2,0}\nA^2_{2,0}+a^4_{2,0}\nA^3_{2,0}.
\eas
By applying this transformation we  have 
\bas
v^{{(0:1)}}_{[4]}&:=&v^{(0:0)}+[\nT_1^1+\nT^1_2,v^{(0:0)}_{[4]}]
\\&=&
\nN+v_1^{{(0:1)}}+v_2^{{(0:1)}}+v_3^{{(0:1)}}.
\eas
In what follows, we  compute  a few structure constants as  an illustration of  Theorem \ref{thm:FSTC3}.

For instance the  Lie bracket of the first term of \(\nT_1^1\)  with the last term of \(v_2^{(0:0)}\)  contains 
\bas
{[}\nA_{1,0}^{0},\nA_{2,0}^{4}{]}&=&
\sum_{\rho=0}^1  \lambda^{0,1}_{4,2}(\rho) \nA^{4-2\rho}_
{3-2\rho,{\rho}}\\
&=& \lambda^{0,1}_{4,2}(0)\nA^{4}_{3,0}+  \lambda^{0,1}_{4,2}(1)\nA^{2}_{1,1}
\\
&=&
 \frac {48\,}{5}\nA^{4}_{3,0}+ \frac{1}{15}\nA^{2}_{1,1}.
\eas
Following \eqref{eq:lambda} we have
\bas
[\nA^0_{1,0},\nA^0_{0,1}]&=&\lambda^{0,1}_{0,0}(0)\nA^0_{1,1}=\nA^0_{1,1},
\\
{[\nA^0_{1,0},\nA^1_{2,0}]}&=&\lambda^{0,1}_{1,2}(0)\nA^1_{3,0}=\frac{2}{3}\nA^1_{3,0},
\\
{[\nA^0_{1,0},\nA^0_{2,0}]}&=&\lambda^{0,1}_{0,3}(0)\nA^0_{3,0}=\nA^0_{3,0},
\\
{[\nA^0_{1,0},\nA^2_{2,0}]}
&=&\lambda^{0,1}_{2,2}(0)\nA^2_{3,0}+\lambda^{0,1}_{2,2}(1)\nA^0_{1,1}=\frac{2}{5}\nA^2_{3,0}+\frac{8}{5}\nA^0_{1,1},
\\
{[\nA^0_{1,0},\nA^3_{2,0}]}&=&
\lambda^{0,1}_{3,2}(0)\nA^3_{3,0}+\lambda^{0,1}_{3,2}(1)\nA^1_{1,1}=\frac{1}{5}\nA^2_{3,0}+\frac{24}{5}\nA^0_{1,1}.
\eas
Note that for computing \(\lambda^{l_1,\mu_1}_{l_2,\mu_2}(0)\) and  \(\lambda^{0,\mu_1}_{l_2,\mu_2}(\rho)\), one can use the 
explicit formulas  of Remarks \ref{rem:lambda0} and  \ref{eq:=lambdal0}, respectively.

Now, we continue the normalization process. Hence, \(v^{(0:1)}_{[4]}\) has the following  components
\bas
v_1^{{(0:1)}}&=&a^0_{1,0}\nA^0_{1,0},
\\
v_2^{{(0:1)}}&=&a^0_{0,1}\nA^0_{0,1}
+a^0_{2,0}\nA^0_{2,0},
\\
v_3^{{(0:1)}}&=&
\left( a^1_{{1,0}}a^0_{{1,1}}+24\,a^1_{{2,0}}a^2_{{3,0}}-6\,a^2_{{1,0}}a^2_{{3,0}}-6\,a^2_{{2,0}}a^2_{{3,0}}+4\,a^3_{{2,0}}a^2_{{3,0}}
+a^0_{{1,1}} \right) \nA^0_{{1,1}}
+a^0_{{3,0}} \left( a^1_{{1,0}}+a^1_{{2,0}}+
1 \right) \nA^0_{{3,0}}
\\&&
+ \left( 24\,a^1_{{1,0}}a^3_{{3,0}}+a^2_{{1,0}}a^1_{{1,1}}+24\,a^2_{{2,0}}a^3_{{3,0}}-16\,a^3_{{2,0}}a^3_{{3,0}}+24\,a_
{{4,2,0}}a^3_{{3,0}}+a^1_{{1,1}} \right) \nA^1_{{1,1}}
\\&&
+a^1_{{3,0}}
\left( a^1_{{1,0}}+a^1_{{2,0}}+a^2_{{1,0}}+a^2_{{2,0}}+1 \right) \nA^1_{{3,0}}
\\&&
+\left(144\,a^1_{{1,0}}a^4_{{3,0}}-36\,a^2_{{1,0}}a^4_{{3,0}}
+24\,a^3_{{2,0}}a^4_{{3,0}}-36\,a^4_{{2,0}}a^4_{{3,0}}+a^2_{{1,1}}\right)\nA^2_{{1,1}}
\\&&
+\left(a^1_{{1,0}}
a^2_{{3,0}}+a^1_{{2,0}}a^2_{{3,0}}+a^2_{{1,0}}a^2_{{3,0}}+a^2_{{2,0}}a^2_{3,0}+a^3_{{2,0}}a^2_{{3,0}}+a^2_{{3,0}}\right)
\nA^2_{{3,0}}
\\&&
+ \left( a^1_{{1,0}}a^3_{{3,0}}+a^2_{{1,0}}a^3_{{3,0}}+a^2_{{2,0}}a^3_{{3,0}}
+a^3_{{2,0}}a^3_{{3,0}}+a^4_{{2,0}}a^3_{{3,0}}+a^3_{{3,0}} \right) \nA^3_{{3,0}}
\\&&
+ \left( a^1_{{1,0}}a^4_{{3,0}}+a^2_{{1,0}}a^4_{{3,0}}+a^3_{{2,0}}a^4_{{3,0}}+a^4_{2,0}a^4_{{3,0}}+a^4_{{3,0}} \right) \nA^4_{{3,0}}
+ \left( a^2_{{1,0}}a^5_{{3,0}}+a^4_{{2,0}}a^5_{{3,0}}+a^5_{{3,0}} \right) \nA^5_{{3,0}}
\\&&
+a^6_{{3,0}}\nA^6_{{3,0}}.
\eas
At the end we apply \(\nT_3^2\) to \(v^{(0:1)}_{[4]}\)  to find  \(v^{(0:2)}_{[4]}.\)
Define
\bas
\nT_3^2&:=&
\left( 24\,a^1_{{1,0}}a^3_{{3,0}}+a^2_{{1,0}}a^1_{{1,1}}+24\,a^2_{{2,0}}a^3_{{3,0}}-16\,a^3_{{2,0}}a^3_{{3,0}}+24\,a^4_{{2,0}}a^3_{{3,0}}+a^1_{{1,1}} \right) \nA^0_{{1,1}}
\\&&
+a^1_{{3,0}}
\left( a^1_{{1,0}}+a^1_{{2,0}}+a^2_{{1,0}}+a^2_{{2,0}}+1 \right) \nA^0_{{3,0}}
\\&&
+\left(144\,a^1_{{1,0}}a^4_{{3,0}}-36\,a^2_{{1,0}}a^4_{{3,0}}
+24\,a^3_{{2,0}}a^4_{{3,0}}-36\,a^4_{{2,0}}a^4_{{3,0}}+a^2_{{1,1}}\right)\nA^1_{{1,1}}
\\&&
+\left(a^1_{{1,0}}
a^2_{{3,0}}+a^1_{{2,0}}a^2_{{3,0}}+a^2_{{1,0}}a^2_{{3,0}}+a^2_{{2,0}}a^2_{3,0}+a^3_{{2,0}}a^2_{{3,0}}+a^2_{{3,0}}\right)
\nA^1_{{3,0}}
\\&&
+ \left( a^1_{{1,0}}a^3_{{3,0}}+a^2_{{1,0}}a^3_{{3,0}}+a^2_{{2,0}}a^3_{{3,0}}
+a^3_{{2,0}}a^3_{{3,0}}+a^4_{{2,0}}a^3_{{3,0}}+a^3_{{3,0}} \right) \nA^2_{{3,0}}
\\&&
+ \left( a^1_{{1,0}}a^4_{{3,0}}+a^2_{{1,0}}a^4_{{3,0}}+a^3_{{2,0}}a^4_{{3,0}}+a^4_{2,0}a^4_{{3,0}}+a^4_{{3,0}} \right) \nA^3_{{3,0}}
\\&&
+ \left( a^2_{{1,0}}a^5_{{3,0}}+a^4_{{2,0}}a^5_{{3,0}}+a^5_{{3,0}} \right) \nA^4_{{3,0}}+a^6_{{3,0}}\nA^5_{{3,0}}.
\eas
Then, we  have
\bas
v^{(0:2)}_{[4]}&:=&v^{(0:1)}+[\nT_3^2,v^{(0:1)}]
\\&=&
\nN+v_1^{(0:1)}+v_2^{(0:1)}+v_3^{(0:2)},
\eas
where 
\bas
v_3^{(0:2)}&:=&\left( a^1_{{1,0}}a^0_{{1,1}}+24\,a^1_{{2,0}}a^2_{{3,0}}-6\,a^2_{{1,0}}a^2_{{3,0}}-6\,a^2_{{2,0}}a^2_{{3,0}}+4\,a^3_{{2,0}}a^2_{{3,0}}
+a^0_{{1,1}} \right) \nA^0_{{1,1}}
\\&&
+a^0_{{3,0}} \left( a^1_{{1,0}}+a^1_{{2,0}}+
1 \right) \nA^0_{{3,0}}.
\eas
Following   Section \ref{3DBT} 
the \(\ker\sad_\nN\)-terms   can not eliminate any terms  from \(v^{(0:2)}_{[4]}\), since the grade is always less than seven. Hence,
 the second  level normal form \(v^{(0:2)}_{[4]}\) is the unique normal form (up till grade three).
\epr
	\section{Concluding remarks}
   We have shown how the inversion formula for the Clebsch-Gordan coefficients (or $3j$-symbols) enables us to determine the structure constants for a basis of the Lie algebra of polynomial vector fields that is adapted to the action of the nilpotent linear vector field in the \(3\)D case. In principle, this is a procedure that could be generalized to higher dimensions, although this will no doubt require some serious work.

  Once the structure constants  have been determined, the analysis of the higher level normal form, and ultimately the unique normal form, comes within reach. Even in the \(2\)D case, 
  this is a nontrivial problem, that will have more and more subcases as the dimension grows.

   In addition,  
  these formulas may prove useful in the study of the convergence of the (first level) normal form where one needs to estimate the 
  transformation  through the Homological Equation. We refer the interested reader to \cite{stolovitch2016holomorphic}, where parts of the proof show a striking similarity to the techniques employed in the present paper and we hope to improve on these results using the techniques described in \cite{mokhtari2019versal}.

	\bibliographystyle{plain}
\bibliography{SP} 

\begin{thebibliography}{10}

\bibitem{baider1991unique}
Alberto Baider and Jan~A Sanders.
\newblock Unique normal forms: The nilpotent {H}amiltonian case.
\newblock {\em Journal of differential equations}, 92(2):282--304, 1991.

\bibitem{baider1992further}
Alberto Baider and Jan~A Sanders.
\newblock Further reduction of the {T}akens-{B}ogdanov normal form.
\newblock {\em Journal of Differential Equations}, 99(2):205--244, 1992.

\bibitem{collingwood1993nilpotent}
David~H Collingwood and William~M McGovern.
\newblock {\em Nilpotent orbits in semisimple Lie algebra: an introduction}.
\newblock CRC Press, 1993.

\bibitem{rationalCGC2018}
Robert~W. Donley, Jr. and Won~Geun Kim.
\newblock A rational theory of {C}lebsch-{G}ordan coefficients.
\newblock In {\em Representation theory and harmonic analysis on symmetric
  spaces}, volume 714 of {\em Contemp. Math.}, pages 115--130. Amer. Math.
  Soc., Providence, RI, 2018.

\bibitem{frenkel2010categorifying}
Igor Frenkel, Catharina Stroppel, and Joshua Sussan.
\newblock Categorifying fractional {E}uler characteristics, {J}ones-{W}enzl
  projector and $3 j $-symbols.
\newblock {\em Quantum Topology}, 3:181--253, 2012.

\bibitem{JF2019vector}
Majid Gazor, Fahimeh Mokhtari, and Jan~A Sanders.
\newblock Vector potential normal form classification for completely integrable
  solenoidal nilpotent singularities.
\newblock {\em Journal of Differential Equations}, 267:3083--3113, 2019.

\bibitem{kassel2012quantum}
Christian Kassel.
\newblock {\em Quantum groups}, volume 155.
\newblock Springer Science \& Business Media, 2012.

\bibitem{kirillov1988quantum}
Anatolii~Nikolaevich Kirillov.
\newblock The quantum {C}lebsch-{G}ordan coefficients.
\newblock {\em Zapiski Nauchnykh Seminarov POMI}, 168:67--84, 1988.

\bibitem{knapp2013lie}
Anthony~W. Knapp.
\newblock {\em Lie groups beyond an introduction}, volume 140.
\newblock Springer Science \& Business Media, 2013.

\bibitem{koornwinder1981clebsch}
Tom~H Koornwinder.
\newblock Clebsch-{G}ordan coefficients for {SU}(2) and {H}ahn polynomials.
\newblock {\em Nieuw Archief voor de Wiskunde (3)}, XXIX:140--155, 1981.

\bibitem{mokhtari2019versal}
Fahimeh Mokhtari and Jan~A Sanders.
\newblock Versal normal form for nonsemisimple singularities.
\newblock {\em Journal of Differential Equations}, 267(5):3083--3113, 2019.

\bibitem{olver1999classical}
Peter~J Olver.
\newblock {\em Classical invariant theory}, volume~44.
\newblock Cambridge University Press, 1999.

\bibitem{sanders2003normal}
Jan~A Sanders.
\newblock Normal form theory and spectral sequences.
\newblock {\em Journal of Differential Equations}, 192(2):536--552, 2003.

\bibitem{SVM2007}
Jan~A Sanders, F.~Verhulst, and J.~Murdock.
\newblock {\em Averaging methods in nonlinear dynamical systems}, volume~59 of
  {\em Applied Mathematical Sciences}.
\newblock Springer, New York, second edition, 2007.

\bibitem{stolovitch2016holomorphic}
Laurent Stolovitch and Freek Verstringe.
\newblock Holomorphic normal form of nonlinear perturbations of nilpotent
  vector fields.
\newblock {\em Regular and Chaotic Dynamics}, 21(4):410--436, 2016.

\bibitem{van2012group}
Bartel L van~der Waerden.
\newblock {\em Group theory and quantum mechanics}, volume 214.
\newblock Springer Science \& Business Media, 2012.

\end{thebibliography}
\end{document}